\documentclass[oneside,11pt]{amsart}
\usepackage{amssymb,amscd,amsmath,latexsym}
\usepackage[mathcal]{euscript}
\usepackage{url}

\newtheorem{thm}{Theorem}[subsection]
\newtheorem{cor}[thm]{Corollary}
\newtheorem{lem}[thm]{Lemma}
\newtheorem{prop}[thm]{Proposition}
\theoremstyle{definition}
\newtheorem{defin}[thm]{Definition}
\newtheorem{exm}[thm]{Example}

\theoremstyle{remark}
\newtheorem{rem}[thm]{Remark}

\newcommand{\Z}{{\mathbb Z}}
\newcommand{\F}{{\mathbb F}}
\newcommand{\Ext}{\operatorname{Ext}}
\newcommand{\Tor}{\operatorname{Tor}}
\newcommand{\End}{\operatorname{End}}
\newcommand{\Hom}{\operatorname{Hom}}

\newcommand{\G}{{\mathcal G}}

\newcommand{\di}{\operatorname{dim}}

\newcommand{\sd}{\Sigma_d}
\newcommand{\vt}{V^{\otimes d}}
\newcommand{\HH}{\operatorname{H}}
\newcommand{\Hu}{\HH^\bullet}
\newcommand{\Hs}{\Hu(\Sigma_d, }
\newcommand{\bu}{\bullet}

 \numberwithin{equation}{subsection}
 \numberwithin{table}{subsection}

\begin{document}

\title[Cohomology for Young Modules]
{\bf On the cohomology of Young modules for the symmetric group}

\author{\sc Frederick R. Cohen}

\address
{Department of Mathematics\\University of Rochester\\Hylan
Building\\Rochester, NY~14627, USA}
 \email{cohf@math.rochester.edu}
 \thanks{Research of the first author was supported in part by NSF
grant DMS-0340575 and DARPA grant 2006-06918-01}

\author{\sc David J. Hemmer}
\address
{Department of Mathematics\\ University at Buffalo, SUNY \\
244 Mathematics Building\\Buffalo, NY~14260, USA}
\thanks{Research of the second author was supported in part by NSF
grant  DMS-0556260} \email{dhemmer@math.buffalo.edu}

\author{\sc Daniel K. Nakano}
\address
{Department of Mathematics\\ University of Georgia \\
Athens\\ GA~30602, USA}
\thanks{Research of the third author was supported in part by NSF
grant  DMS-0654169} \email{nakano@math.uga.edu}

\date{April 2009}


\begin{abstract}


The main result of this paper is an application of the topology of
the space $Q(X)$ to obtain results for the cohomology of the symmetric
group on $d$ letters, $\Sigma_d$, with `twisted' coefficients in various choices of
Young modules and to show that these computations reduce to certain
natural questions in representation theory. The authors extend
classical methods for analyzing the homology of certain spaces
$Q(X)$ with mod-$p$ coefficients to describe the homology
$\HH_{\bullet}(\Sigma_d, V^{\otimes d})$ as a module for the general
linear group $GL(V)$ over an algebraically closed field $k$ of
characteristic $p$. As a direct application, these results provide a
method of reducing the computation of
$\text{Ext}^{\bullet}_{\Sigma_{d}}(Y^{\lambda},Y^{\mu})$ (where
$Y^{\lambda}$, $Y^{\mu}$ are Young modules) to a representation
theoretic problem involving the determination of tensor products and
decomposition numbers. In particular, in characteristic two, for
many $d$, a complete determination of $\Hs Y^\lambda)$ can be found.
This is the first nontrivial class of symmetric group modules where
a complete description of the cohomology in all degrees can be
given.

For arbitrary $d$ the authors determine $\HH^i(\Sigma_d,Y^\lambda)$
for $i=0,1,2$. An interesting phenomenon is uncovered--namely a
stability result reminiscent of generic cohomology for algebraic
groups. For each $i$ the cohomology $\HH^i(\Sigma_{p^ad},
Y^{p^a\lambda})$ stabilizes as $a$ increases. The methods in this
paper are also powerful enough to determine, for any $p$ and
$\lambda$, precisely when $\HH^{\bullet}(\sd,Y^\lambda)=0$. Such
modules with vanishing cohomology are of great interest in
representation theory because their support varieties constitute the
representation theoretic nucleus.
\end{abstract}

\maketitle

\section{Introduction}
\subsection{} The representation theory of the symmetric group and its
connections with the general linear group have been studied for more
than 100 years. Yet there exist very few calculations of cohomology
groups $\Ext^i_{k\Sigma_d}(M,N)$ for natural modules  $M$ and $N$
for the symmetric group $\Sigma_d$ over an algebraically closed
field $k$ of characteristic $p$. The cohomology ring
$\HH^{\bullet}(\Sigma_d,k)\cong \Ext^{\bullet}_{k\sd}(k,k)$ was
originally computed (as a vector space) by Nakaoka \cite{Nakaoka}.
Feshbach provided a combinatorial formulation of the ring structure
over $\F_2$ \cite{feshbach}. However for odd primes the ring
structure is only fully understood for small $d$. Using the results
of Araki-Kudo and Dyer-Lashof one can compute the homology
$\HH_{\bullet}(\Sigma_d,V^{\otimes d})$ over a field with $p$
elements for $p$ prime \cite{ak,dl}. The aforementioned computations
all employ constructions and techniques from algebraic topology.

In the late 1990's, Doty, Erdmann and Nakano \cite{DEN} introduced a
first quadrant spectral sequence which related the cohomology for
the symmetric group to that of the general linear group (cf.
\eqref{eq: spectralsequenceTorversion}). Let $V \cong k^n$ be the
natural module for the general linear group $GL(V) \cong GL_n(k)$.
The construction utilizes the commuting actions of $GL(V)$ and
$\Sigma_d$ on $V^{\otimes d}$ in a functorial way. The relationship
between the two cohomology theories relies heavily on understanding
the structure of $\text{Ext}^{j}_{k\Sigma_{d}}(V^{\otimes d},N)$ as
a $GL(V)$-module. Vanishing ranges for these cohomology groups were
obtained in \cite{KNcomparingcoho}, for $N$ a dual Specht or Young
module, which enabled the proofs of stability results between the
cohomology of certain $GL(V)$-modules and $\Sigma_{d}$-modules.
Beside these examples, the only known computations are for small
$d$. For Specht modules $S^\lambda$, even
$\Ext^1_{\sd}(k,S^\lambda)$ is unknown. When $D_\lambda$ and $D_\mu$
are simple modules, $\Ext^1_{\sd}(D_\lambda,D_\mu)$ is known in some
very particular examples.

\subsection{} In this paper we aim to effectively combine both
techniques mentioned above to calculate extension groups and
cohomology for the symmetric group. This idea is quite natural
because the structure of these Ext groups arises from
interconnections of topology and algebra as described above. An
example of these connections is  provided in a beautiful recent
paper of Benson \cite{benson}, where he shows that the homology of
the loop space of the $p$-completion of $BG$, denoted $\Omega({BG
^\wedge_p})$, depends only on representation theoretic information.
One crude overlap with the results here is that the homology groups
$\HH_{\bullet}(\Sigma_d,  V^{\otimes d})$ are also given in terms of
the homology of  certain loop spaces.

A common theme that runs throughout this paper is that symmetric
group cohomology is reduced to $GL_{n}(k)$-representation theory,
i.e., a homological version of Schur-Weyl duality. Our first step
involves applying the aforementioned spectral sequence to calculate
cohomology of Young modules. In this case the spectral sequence
collapses, and one needs to only understand the simple constituents
of $\HH_{\bullet}(\Sigma_d,\vt)$ as a $GL_n(k)$-module. This
demonstrates that homological information can be reduced to
questions where purely representation theoretic information arises.
Our Theorem \ref{thm:shiftedfinalanswer for p=2 and p oddL} gives a
complete description of this module as a tensor product of twisted
symmetric and exterior powers of $V$.

The computation of $\HH_{\bullet}(\Sigma_d,\vt)$ as a
$GL_n(k)$-module leads to a description of Young module cohomology
$\Hs Y^\lambda)$ in terms of the simple constituents of these
modules, which can often be completely calculated. One feature of
the result is that computing symmetric group cohomology in
arbitrarily high degree with coefficients in a Young module is
reduced to determining the composition factors of certain natural
modules for the general linear group.

For example, this method can be used easily and directly to
determine $\HH^i(\Sigma_{16}, Y^\lambda)$ in characteristic two for
any $i \geq 0$ and any $\lambda \vdash 16$. A useful fact that
we prove is when the decomposition matrices are known, a complete
description of $\Ext^i_{\Sigma_d}(Y^\lambda, Y^\mu)$ can be given.

With our description of $\HH_{\bullet}(\Sigma_d,\vt)$ as a
$GL_n(k)$-module in terms of twisted symmetric and exterior powers,
we invoke the work of Doty on the composition factors of symmetric
powers to determine precisely, in any characteristic, which
$\lambda$ have the property that $\Hs Y^\lambda)$ is identically
zero. Such modules in the principal block with no cohomology can be
used to describe the {\it representation theoretic nucleus}. We can
also compute $\HH^i(\Sigma_d, Y^\lambda)$ for small degrees $i$ and
arbitrary $d$ and $\lambda$.

\subsection{}It has always been a mystery how the Frobenius morphism on the group scheme $GL_{n}$
plays a role in the cohomology theory of the symmetric group.
Certainly, our description of $\HH_{\bullet}(\Sigma_d,\vt)$ as a
$GL_n(k)$-module is a good start to understanding this phenomenon.
We prove a stability result relating cohomology of Young modules
$Y^{p^a\lambda}$ and $Y^{p^{a+1}\lambda}$. This appears to be the
first result for the representation theory of the symmetric group
which involves multiplying a partition by $p$. This is further
evidence that the Frobenius will play an integral role in our
understanding of symmetric group representations, even though there
is no obvious interpretation of twisting representations on the
symmetric group side (as opposed to twisting
$GL_{n}$-representations).

\begin{rem}The main result, Theorem \ref{thm:shiftedfinalanswer for p=2 and p oddL}, interpreting Young module
cohomology in terms of composition factors of certain
$GL_n(k)$-modules, holds in any characteristic. However some of the
subsequent results and computations are only for characteristic two
while others are general, so we will be careful to state in which
characteristic the various results apply.
\end{rem}

\subsection{Organization}

The paper is organized in the following manner. Section
\ref{section: The Schur and inverse Schur functor} is an exposition
of the relationship between $\HH_{\bullet}(\Sigma_d, \vt)$ as a
$GL_n(k)$-module and the cohomology of Young modules based on
results from \cite{DEN}. In Sections \ref{sec: On the homology of
$Q(X)$}-\ref{sec:degree.shift} a natural construction from
algebraic topology is presented which gives a complete description of
$\HH_i(\Sigma_d, \vt)$ as a $GL_n(k)$-module.

In Section \ref{sec: cohomology between Young modules} we provide
a list of sufficient conditions on $\HH^\bu(\Sigma_d, Y^\lambda)$ to determine
$\Ext^\bu_{\Sigma_d}(Y^\lambda, Y^\mu)$. Section \ref{sec: completeexampled=6}
gives an application of this machinery to provide complete answers for $\Sigma_6$ in
characteristic two and explains how one can easily replicate this
for all $\Sigma_d$, $d \leq 16$.

The subject of Section \ref{sec: Cohomology vanishing theorems} is a
precise determination of which Young modules have vanishing
cohomology for arbitrary $\Sigma_d$ and arbitrary characteristic.
In Section \ref{sec: low degree cohomology} a
complete calculation of Young module cohomology is given in
degrees $0,1, 2$ for arbitrary $\Sigma_d$ in characteristic two.

The stability theorem discussed above is proven in Section \ref{sec:
Generica cohomology for Young modules}. The final section is a description of
similar results for cohomology of permutation modules $M^\lambda$. In this case the answer is given
not in terms of composition factor multiplicities of an explicit
$GL_d(k)$-module, but instead in terms of weight-space
multiplicities of the same module.

The authors would like to express their appreciation to the referee
whose comments were very helpful in improving an earlier version of
this manuscript.

\section{Homology and Commuting Actions}
\label{section: The Schur and inverse Schur functor}

\subsection{}

In this section we will explain why cohomology of Young modules is
related to the homology groups $\HH_{\bullet}(\Sigma_d, \vt)$. We
will assume the reader is familiar with the definition of Young
modules for the symmetric group, in particular determining which are
nonprojective and which lie in the principal block (cf.
\cite[4.6]{Martinbook} for details).

Let $V$ be the vector space of $n$-dimensional column vectors over
the field $k$.
\begin{rem}\label{rem:choices.of.fields}
Some of the results below apply to any field $k$ while
$\overline{\F}_p$ is required in most of the work below. In
addition, there are points where it is useful to use $\F_p$. Thus it
will be stated explicitly when the field $k$ will be restricted to
either $ \overline{\F}_p$  or $\F_p$.
\end{rem}

The general linear group $GL_n(k):=GL(V)$ acts naturally on $V$, and
thus on $\vt$. This action commutes with the action of the symmetric
group $\sd$, acting by place permutation. Thus there is a map from
$kGL_n(k)$ into the {\it Schur algebra}
$$S(n,d)= \End_{k\sd}(\vt).$$ The famous double centralizer theorem \cite[2.6c]{GreenPolynoGLn} states, in part, that this map is a surjection.

We briefly describe the setup from \cite{DEN}. When $n \geq d$ there
is an idempotent $e \in S(n,d)$ such that $eS(n,d)e \cong k\sd$.
This gives an exact covariant functor ${\mathcal F}$ going from
$\text{Mod}(S(n,d))$ to $\text{Mod}(k\sd)$ defined by
${\mathcal F}(M) = eM$ and called
the Schur functor. The Schur functor can be realized as both a
$\Hom$ and a tensor product functor:
$${\mathcal F}(M)=eM \cong \Hom_{S(n,d)}(S(n,d)e,M) \cong eS(n,d) \otimes_{S(n,d)} M.$$
Thus it admits two natural (one-sided) adjoint functors from mod-$k\sd$ to mod-$S(n,d)$:
    $$\G_{\Hom}(N) = \Hom_{\sd}(eS(n,d),N), \,\,\,\,\,\G_\otimes(N)=S(n,d)e \otimes_{eS(n,d)e} N.$$
The functor $\G_{\Hom}$ is left exact, and so admits higher right
derived functors $R^\bullet \G_{\Hom}$. The functor $\G_\otimes$ is
right exact and so admits higher left derived functors $L_\bullet
\G_\otimes.$ These derived functors may  be expressed as:
$$R^j \G_{\Hom}(N) = \Ext^j_{k\sd}(\vt,N), \,\,\,\,\, L_j\G_\otimes(N)=        \Tor_j^{k\sd}(\vt,N).$$

    In \cite{DEN} two first-quadrant Grothendieck spectral
sequences are constructed that relate cohomology for $GL_n(k)$ to
that of $k\sd$. For example, the pair of functors ${\mathcal F}$ and
$\G_\otimes$ gives rise to a Grothendieck spectral sequence with
$E_2$ page:

\begin{equation}\label{eq: spectralsequenceTorversion}
E_2^{i,j}=\Ext^i_{S(n,d)}(\Tor_j^{\sd}(\vt, N),M) \Rightarrow
\Ext^{i+j}_{\sd}(N,eM).
\end{equation}

\subsection{} In Section~\ref{sec: On the homology of $Q(X)$}, we will
focus on the computation of
$$\HH_\bu(\Sigma_d,  V^{\otimes d})=\Tor_{\bullet}^{k\sd}(k,\vt).$$
Historically, topologists have viewed this homology computation as
more natural than the calculation of the cohomology. This might be
due to the fact that one can express the homology in terms of
symmetric and exterior powers. On the other hand, algebraists have
often viewed cohomology as more natural for applications in
representation theory. This section is devoted to keeping track of
the various $GL_{n}(k)$-actions on the homology and cohomology
groups.

Normally, one would consider $\vt$ as a
$\Sigma_{d}$-$GL_{n}(k)$-bimodule and view the tensor space as a
right $GL_{n}(k)$-module. However, the action of the symmetric group
and the general linear group commute and we can view $\vt$ as a left
$\Sigma_{d}\times GL_{n}(k)$-module (the $\Sigma_{d}$-action is by
place permutation and the $GL_{n}(k)$-action is by the diagonal action).
This will induce a left $GL_{n}(k)$-action on $\HH_\bu(\Sigma_d,  V^{\otimes d})$.

Next observe that, by twisting by the map on $k\Sigma_d$ given by
$\sigma \rightarrow \sigma^{-1}$, one can make any left
$k\Sigma_{d}$-module into a right $k\Sigma_{d}$-module and vice
versa. Using this correspondence, one can see that
$\Tor_{\bullet}^{k\sd}(M,N)\cong \Tor_{\bullet}^{k\sd}(N,M)$. In
particular,
$$L_{\bu}\G_\otimes(k)=\Tor_{\bu}^{k\sd}(\vt,k)\cong \Tor_{\bu}^{k\sd}(k,\vt)\cong \HH_\bu(\Sigma_d,  V^{\otimes d}).$$
Moreover, by tracing through the definitions one sees that these are
isomorphisms of left $GL_{n}(k)$-modules.

For an $S(n,d)$-module $M$, let $M^\tau$ denote the contravariant
(or transpose)  dual of $M$, as described in
\cite[2.7]{GreenPolynoGLn}. The underlying vector space of
$M^{\tau}$ is $\text{Hom}_{k}(M,k)$ and the action of $GL_n(k)$ is
\cite[2.7a]{GreenPolynoGLn} $(g.f)(m)=f(g^{\operatorname{tr}}.m)$
for $f\in M^{\tau}$, $m\in M$, and $g\in GL_{n}(k)$. Let $*$ be the
contragredient dual on $k\Sigma_{d}$ modules. By applying the change
of rings formula, one has
\begin{eqnarray*}
[V^{\otimes d}\otimes_{k\Sigma_{d}}(-)^{*}]^{\tau}&\cong& \text{Hom}_{k}(V^{\otimes d}\otimes_{k\Sigma_{d}} (-)^{*},k)\\
&\cong& \text{Hom}_{k\Sigma_{d}}(V^{\otimes d},\text{Hom}_{k}((-)^{*},k))\\
&\cong& \text{Hom}_{k\Sigma_{d}}(V^{\otimes d},(-)).
\end{eqnarray*}
In the last line we are using the fact that $\text{Hom}_{k}((-)^{*},k))\cong ((-)^{*})^{*}\cong (-)$.
This induces a natural isomorphism of functors:
$R^{\bullet}\G_{\Hom}(-)\cong
L_{\bullet}\G_\otimes((-)^{*})^{\tau}$. In particular as left
$GL_{n}(k)$-modules
\begin{equation}
\label{eq:relateRjLj} R^{j}\G_{\Hom}(k)\cong
L_j\G_\otimes((k)^{*})^{\tau}\cong \HH_{j}(\Sigma_{d},V^{\otimes
d})^{\tau}
\end{equation}
for all $j\geq 0$. Since the $\tau$-duality fixes all simple
$S(n,d)$-modules, the $GL_{n}(k)$-composition factors of
$R^{j}\G_{\Hom}(k)$ are the same as in
$\HH_{j}(\Sigma_{d},V^{\otimes d})$.

\subsection{} The simple $S(n,d)$-modules for $n\geq d$ are parameterized by partitions of $d$. When $\lambda$ is
a partition of $d$ (denoted by $\lambda \vdash d$), let $L(\lambda)$
be the corresponding simple $S(n,d)$-module and $I(\lambda)$ be its
injective hull in the category of $S(n,d)$-modules. Note that the
simple $S(n,d)$-modules are exactly the simple polynomial
$GL_{n}(k)$-modules of degree $d$. For a description of the simple
$GL_n(k)$-modules and the Steinberg Tensor Product Theorem (which
will be used repeatedly), the reader is referred to \cite{Jantzen}.
If $M$ is a finite-dimensional $GL_{n}(k)$-module then let $[M:S]$
denote the multiplicity of the simple module $S$ in a composition
series of $M$. Under the Schur functor, $I(\lambda)$ maps to the
Young module $Y^\lambda$. If one sets $M:=I(\lambda)$ then the
spectral sequence \eqref{eq: spectralsequenceTorversion} collapses
to get part (a) of the following theorem. Part (b) is obtained by
further specializing $N:=k$ and using the fact that both $\vt$ and
$Y^\lambda$ are self-dual $k\sd$-modules while $L(\lambda) \cong
L(\lambda)^\tau$, together with \eqref{eq:relateRjLj}.

\begin{thm}\cite[Prop 2.6B]{DEN} \label{thm:multl of simples in Ljgivescohomology}
\begin{itemize}
\item[(a)] $\dim_k\Ext^i_{\sd}(N,Y^\lambda)=[L_i\G_\otimes(N):L(\lambda)], \,\, i \geq 0.$
\item[(b)] $\dim_k \HH^i(\Sigma_d,Y^\lambda) = [\HH_i(\Sigma_d,\vt):L(\lambda)], \,\, i \geq 0.$
\end{itemize}
\end{thm}

The theorem above indicates that determining the simple constituents
of $\HH_{\bullet}(\sd, \vt)$ as a graded $GL_n(k)$ module allows one
to calculate Young module cohomology in all degrees. At the time of
\cite{DEN} there seemed to be no way to understand even the
dimension of $\HH_i(\sd,\vt)$, let alone its $GL_n(k)$-module
structure, even in the case $i=1$. In the following sections some
results from algebraic topology are used to give a complete,
explicit description of $\HH_{\bullet}(\sd, \vt)$ as a
$GL_n(k)$-module. In particular, Theorem \ref{thm:shiftedfinalanswer
for p=2 and p oddL} and Corollary \ref{corgivingstructure} imply that the module is just a direct sum of
tensor products of Frobenius twists of symmetric and (for $p$ odd)
exterior powers of the natural module $V$.

\section{On the homology of $Q(X)$ and $\HH_\bu(\Sigma_d,  V^{\otimes d})$}\label{sec: On the homology of $Q(X)$}
\subsection{}

Fix a field $k$ together with a graded vector space $V$ over $k$
where $V$ is concentrated in degrees strictly greater than $0$ (i.e.,
$V$ is $\mathbb N$-graded). Such vector spaces are called {\it
connected}. This choice of grading is a formal convenience
concerning $V$ which reflects the natural topological setting used
below. Properties of the homology groups $\HH_\bullet(\Sigma_d,
V^{\otimes d})$ are addressed below.

The purpose of this section is to describe the functor given by the
direct sum $ \oplus_{d \geq 0}\HH_\bullet(\Sigma_d, V^{\otimes d})$
in terms of the known singular homology of a topological space with
coefficients in a field $k$. Some of the results described below
hold for any field $k$ while some depend on the choice of field.

Most of the results used below hold in case $k$ is either $\F_p$,
the field with $p$ elements for $p$ prime or the algebraic closure
of $\F_p$, denoted $\overline{\F}_p$. In these cases below, the
specific choice of field, either $\F_p$ or $\overline{\F}_p$, will
be indicated.

The results presented next are classical. Our intent is to provide a
clear self-contained exposition while keeping track of additional
new data, in particular the $GL(V)$-module structure on the homology
with coefficients in $k = \overline{\F}_p$. With this information,
we show that a combination of known algebraic and topological
results admits new applications within representation theory. The
authors intend, in a future paper, to approach related questions for
certain choices of Hecke algebras.

\subsection{} Assume that $V$  is
given in degree $s$ by a vector space $V_s$ having dimension $b_s <
\infty$. Thus $V$ is a graded vector space which may be non-trivial
in arbitrarily many degrees, but is required to be of finite
dimension in any fixed degree. To coincide with standard topological
constructions, say that such a $V$ has finite type. It is not
necessary to make this assumption in what follows below concerning
homology groups. However, this assumption is useful in passage to
cohomology groups.

In addition as remarked above, $V$ is assumed to be connected.
Notice that by natural degree shift arguments (cf. Section
\ref{sec:degree.shift}.),  it suffices to consider $\oplus_{d \geq
0}\HH_\bullet(\Sigma_d, V^{\otimes d})$ for connected vector spaces
$V$  (assuming that a given vector space is totally finite
dimensional as a graded vector space).

Let  $\hbox{sgn}$ denote the one-dimensional sign representation of
the symmetric group. One feature which arises by carrying out the
work below in the context of graded vector spaces, rather than
ungraded vector spaces, is that the methods also apply directly to
the case of the $\Sigma_d$-module $\vt \otimes \hbox{sgn}$. Thus
these techniques should apply to a calculation of cohomology of
twisted Young modules $Y^\lambda \otimes \hbox{sgn}$, and even to
the so-called {\it signed} Young modules, which were defined by
Donkin in \cite{DonkinSymmetric}.

If $V$ is a connected vector space of finite type over any field
$k$, then $ \oplus_{d \geq 0}\HH_\bullet(\Sigma_d, V^{\otimes d})$
is isomorphic to the singular homology of a certain topological
space studied in other contexts first arising in work of Nakaoka,
Steenrod, Araki-Kudo, Dyer-Lashof and others \cite{ak, dl,Nakaoka,
s}. That is the topological space $Q(X)$ described in the next
section.

\subsection{}

Recall that the reduced homology groups of the $n$-sphere $S^n$ for
$n > 0$  $$\bar H_i(S^n,k)$$  are all $0$ for $i \neq n$ and is the
field $k$ in case $i = n$. In addition, a wedge of spheres $$\vee_{n
\in W} S^n$$ for an index set $W$ has the property that there is an
isomorphism
$$\bar H_i(\vee_{n \in W} S^n,k) \to \oplus_{n \in W}\bar{H}_i(S^n,k).$$

Thus given a graded, connected vector space $V$, there is a choice
of a wedge of spheres $$X = \vee_{n \in W} S^n$$ such that there is
an isomorphism $\overline{\HH}_\bu(X,k) \to V$, where
$\overline{\HH}_\bu(X,k)$ denotes the reduced homology of $X$ with
coefficients in $k$. In particular if $V$ has basis $\{v_1(q),
\cdots, v_{b_q}(q)\}$ in degree $ q > 0$, then $X$ may be chosen to
be $$X=\bigvee_{1 \leq q < \infty} (\bigvee_{b_q} S^q).$$

The space $Q(X)$ is defined as $$Q(X)=\bigcup_{ 0<m < \infty}
\Omega^m\Sigma^m(X)$$ where $\Omega^m\Sigma^m(X)$ denotes the space
of continuous, pointed functions from the $m$-sphere to the $m$-fold
suspension of the space $X$. The space $Q(X)$ has been the subject
of thorough investigation and admits many applications \cite{ak,clm,
dl}.

Next, restrict attention to the fields $\overline{\F}_p$ and $\F_p$.
Notice that the universal coefficient theorem gives a natural
isomorphism $$\rho: \HH_\bullet(Y, \F_p)\otimes_{\F_p}
\overline{\F}_p \to \HH_\bullet(Y,\overline{\F}_p)$$ for
path-connected spaces of the homotopy type of a CW complex of finite
type. Thus the homology of $Q(X)$ satisfies the property that there
are isomorphisms $$\rho: \HH_\bullet(Q(X), \F_p)\otimes_{\F_p}
\overline{\F}_p \to \HH_\bullet(Q(X),\overline{\F}_p).$$

\subsection{} A classical result due to Araki-Kudo and Dyer-Lashof
is described first  where homology is taken with coefficients in
$\F_p$ \cite{ak,clm, dl}. Their results are then developed below
both in Definition \ref{defin:operations} and Theorem
\ref{thm:araki.kudo.dyer.lashof.operations.over.overline.F} to give
the results required here with coefficients in $\overline{\F}_p$.

\begin{thm}\label{thm:araki.kudo.dyer.lashof}
Let $X$ denote a path-connected CW-complex with $V =
\overline{\HH}_\bullet(X, \F_p).$ There is a natural isomorphism
$$H: \HH_\bullet(Q(X),\F_p) \longrightarrow  \oplus_{d \geq
0}\HH_\bullet(\Sigma_d, V^{\otimes d}).$$ Furthermore,
$\HH_\bullet(Q(X),\F_p)$ is a known, explicit functor of $
\HH_\bullet(X,\F_p)$ described below.
\end{thm}

Therefore, the universal coefficient theorem together with Theorem
\ref{thm:araki.kudo.dyer.lashof} has the following consequence,
where some additional structure is required to give the
$GL(V)$-action.

\begin{thm}\label{thm:homology.splitting}
Let $X$ denote a path-connected CW-complex with $V =
\overline{\HH}_\bullet(X, \overline{\F}_p).$ There is a natural
isomorphism
$$\overline{H}: \HH_\bu(Q(X),\overline{\F}_p)
\longrightarrow \oplus_{d \geq 0}\HH_\bu(\Sigma_d, V^{\otimes d}).$$
Furthermore, the natural action of $GL(V)$ on $\oplus_{d \geq
0}\HH_\bu(\Sigma_d, V^{\otimes d})$ is described below.
\end{thm}

\begin{rem}\label{remark:spaces}
\begin{enumerate}
  \item A refined version of Theorem \ref{thm:araki.kudo.dyer.lashof} is
stated below as Theorem \ref{thm:astable.splitting}.
  \item The algebraic decomposition implied by Theorems
\ref{thm:araki.kudo.dyer.lashof.operations.over.overline.F},
\ref{thm:homology.splitting}, and \ref{thm:araki.kudo.dyer.lashof}
correspond precisely to a geometric decomposition of the space
$Q(X)$, at least after sufficient suspensions (as given in Theorem
\ref{thm:astable.splitting}) \cite{cmt, k,snaith}.
  \item The main new ingredient here is the action of $GL(V)$; that action is described in Sections
\ref{sec:The.prime.two} and \ref{sec:Odd.primes} in terms of
operations known either as Dyer-Lashof operations or
Araki-Kudo-Dyer-Lashof operations \cite{ak,dl} with some properties
listed next.
\end{enumerate}

\end{rem}

Recall the Araki-Kudo-Dyer-Lashof operations \cite{ak,clm, dl} given
by homomorphisms $$Q_i: \HH_s(Q(X),\F_2) \to
\HH_{i+2s}(Q(X),\F_2),$$ and if $p$ is odd, $$Q_i: \HH_s(Q(X),\F_p)
\to \HH_{i(p-1)+ps}(Q(X),\F_p)$$ with $s+i \equiv 0(\text{mod }2)$.
The $Q_i$ are linear maps over $\F_p$.

\begin{defin}\label{defin:operations}
There are functions $$\overline{Q}_i: \HH_s(Q(X),\overline{\F}_2)
\to \HH_{i+2s}(Q(X),\overline{\F}_2),$$  and if $p$ is odd,
$$\overline{Q}_i: \HH_s(Q(X),\overline{\F}_p) \to
\HH_{i(p-1)+ps}(Q(X),\overline{\F}_p)$$  defined by the formula
$\overline{Q}_i(x) = \rho(Q_i(x) \otimes 1).$
\end{defin}

\begin{rem}\label{remark:non.linearity}
The functions $\bar{Q}_i$ are additive, but not linear. The failure
of linearity is the content of the next theorem which will be used to determine the
$GL(V)$-action.
\end{rem}

\begin{thm}\label{thm:araki.kudo.dyer.lashof.operations.over.overline.F}
Let $x, y$ denote elements in $\HH_s(Q(X),\overline{\F}_p)$. The
functions $\overline{Q}_i$ defined above satisfy the following
properties.

\begin{itemize}
\item[(a)] $\overline{Q}_i(x+y) = \overline{Q}_i(x) + \overline{Q}_i(y)$,
\item[(b)] $\overline{Q}_i(\lambda x) = (\lambda)^p \overline{Q}_i(x)$ for $\lambda \in \overline{\F}_p$.
\end{itemize}

Let $V = \overline{\HH}_\bullet(X, \overline{\F}_p)$. The action of
$GL(V)$ on $\oplus_{d \geq 0}\HH_\bu(\Sigma_d, V^{\otimes d})$ is
derived from these formulas in Section \ref{sec:The.prime.two} in
case $p = 2$ and Section \ref{sec:Odd.primes} in case $p$ is odd.
\end{thm}

The statements in this theorem follow at once from classic work
\cite{ak,dl} and the formula $\overline{Q}_i(x) = \rho(Q_i(x)
\otimes 1)$ with the sole exception of the statement
$\overline{Q}_i(\lambda x) = (\lambda)^p \overline{Q}_i(x)$ for
$\lambda \in \overline{\F}_p$. A proof of this last statement is
given in Lemma \ref{lem:scalars} below.

The explanation for why Theorem
\ref{thm:araki.kudo.dyer.lashof.operations.over.overline.F} suffices
to specify the $GL(V)$-action on $\oplus_{d \geq 0}\HH_\bu(\Sigma_d,
V^{\otimes d})$ is the subject of the next subsection.

\subsection{}\label{sec:explanation.of.the.connection}
This section provides a sketch of the connection between Theorem
\ref{thm:araki.kudo.dyer.lashof.operations.over.overline.F}  and the
$GL(V)$-action on $\oplus_{d \geq 0}\HH_\bu(\Sigma_d, V^{\otimes
d})$.

Let $V = \overline{\HH}_\bullet(X, \overline{\F}_p)$. The direct sum
$\oplus_{d \geq 0}\HH_\bu(\Sigma_d, V^{\otimes d})$ is naturally an
algebra generated by elements in $V$ and compositions of the
operations $\overline{Q}_i(-)$. Thus to identify the $GL(V)$-action,
it suffices to identify the induced action on

\begin{enumerate}
  \item elements $v \in V$ together with the elements $Q_I(v)$ where $Q_I(-)$
  denotes compositions of the operations
$\overline{Q}_i(-)$, and
  \item the action on products of the $Q_I(v)$.
\end{enumerate}

Thus the results in Theorem
\ref{thm:araki.kudo.dyer.lashof.operations.over.overline.F} suffice
to give the requisite commutation formulas with elements in $GL(V)$
given by \begin{enumerate}
\item $\overline{Q}_i(x+y) = \overline{Q}_i(x) + \overline{Q}_i(y)$,
and
\item $\overline{Q}_i(\lambda x) = (\lambda)^p \overline{Q}_i(x)$ for $\lambda \in \overline{\F}_p$.

\end{enumerate}

In Section \ref{sec:The.prime.two} and Section \ref{sec:Odd.primes}
below, $ \oplus_{d \geq 0}\HH_\bu(\Sigma_d, V^{\otimes d})$ is given
by

\begin{enumerate}
  \item  a polynomial ring with generators given by iterates
  of the operations $\overline{Q}_i$ applied to elements $x \in V = \overline{\HH}_\bullet(X,
  \overline{\F}_p)$ for $p = 2$, and

  \item a tensor product of a polynomial ring with an exterior algebra with generators given by iterates
  of the operations $\overline{Q}_i$ as well as Bocksteins applied to elements $x \in V = \overline{\HH}_\bullet(X,
  \overline{\F}_p)$ for $p > 2$.
\end{enumerate}

Some additional information is given next concerning the connection
between the operations above and elements in $\HH_\bu(\Sigma_d,
V^{\otimes d})$.

\begin{enumerate}
  \item Assume that $v \in V = \overline{\HH}_\bullet(X,
\overline{\F}_p)$. Then $\overline{Q}_i(v)$ is an element in
$\HH_\bu(\Sigma_p, V^{\otimes p})$.
  \item Assume that $w \in \HH_\bu(\Sigma_d, V^{\otimes d})$, then
$\overline{Q}_i(w)$ is an element in $\HH_\bu(\Sigma_{pd},
V^{\otimes pd})$.
  \item Assume that $u_1 \in \HH_\bu(\Sigma_{d_1}, V^{\otimes d_1})$, and
$u_2 \in \HH_\bu(\Sigma_{d_2}, V^{\otimes d_2})$, then the product
$u_1\cdot u_2$ is an element in $\HH_\bu(\Sigma_{(d_1+ d_2)},
V^{\otimes (d_1+ d_2)})$.
\end{enumerate}

Thus to give the $GL(V)$-action, it suffices to identify this action
on the elements of $V$ and to evaluate the extension of this action
to composites of the $\overline{Q}_i$ and their products as
specified inductively by Theorem
\ref{thm:araki.kudo.dyer.lashof.operations.over.overline.F}. This
process is carried out in finer detail in the next two sections.

\section{The prime $2$}\label{sec:The.prime.two}
\subsection{}

The purpose of this section is to describe the known homology groups
$\HH_\bu(Q(X),\overline{\F}_2)$ and then to describe the
$GL(V)$-action on $\HH_\bu(Q(X),\overline{\F}_2)$.

Let $\{v_{\gamma}\}$ denote a choice of basis for $V =
\overline{\HH}_\bu(X, \F_2)$. Then there is an isomorphism of
algebras
$$\HH_\bu(Q(X),\F_2) \to S[Q_I(v_{\gamma})]$$
where the following hold.

\begin{enumerate}
\item The algebra $S[Q_I(v_{\gamma})]$ denotes the polynomial algebra with
generators $Q_I(v_{\gamma})$.
\item The generators $Q_I(v_{\gamma})$ are specified
by $I = (i_1,i_2,\cdots, i_t)$ with $$0 < i_1 \leq i_2 \leq \cdots
\leq i_t < \infty, $$ and $$Q_I(v_{\gamma})= Q_{i_i} Q_{i_2}\cdots
Q_{i_t}(v_{\gamma}).$$
\item The empty sequence $I=\emptyset$ is allowed and in this case
$Q_{\emptyset}(v_{\gamma})$ is defined to be $v_{\gamma}$.

\item The weight of a product is defined by $w(X\cdot Y ) = w(X) +w(Y)$.
\end{enumerate}
Define the weight of a monomial $Q_{I}(v_{\gamma})$ where $I =
(i_1,i_2,\cdots, i_t)$ to be $w(Q_{I}(v_{\gamma})) = 2^t.$ A basis
for the symmetric algebra is given by choices of the products of
monomials
$$\mathcal B(V)=\{Q_{I_1}(v_{\gamma_1}) \cdots
Q_{I_k}(v_{\gamma_s})\}.$$  In addition, a basis for
$\HH_\bu(\Sigma_d, V^{\otimes d})$ are those monomials in $\mathcal
B(V)$ with
$$w(Q_{I_1}(v_{\gamma_1})) + \cdots +
w(Q_{I_s}(v_{\gamma_s})) = d$$ as given in \cite{clm}.

\subsection{} The required modification for coefficients in $\overline{\F}_2$ is
stated next. These follow at once by tensoring with
$\overline{\F}_2$, appealing to the universal coefficient theorem
and quoting Theorem
\ref{thm:araki.kudo.dyer.lashof.operations.over.overline.F}
concerning the $GL(V)$-action.

Let $\{v_{\gamma}\}$ denote a choice of basis for the reduced
homology $V = \overline{\HH}_\bu(X,\overline{\F}_2)$. Then there is
an isomorphism of algebras

$$\HH_\bu(Q(X),\overline{\F}_2) \to S[\overline{Q}_I(v_{\gamma})]$$
where the following hold.

\begin{enumerate}
  \item The algebra $S[\overline{Q}_I(v_{\gamma})]$ denotes the polynomial algebra with
  generators $\overline{Q}_I(v_{\gamma})$.
  \item The generators $\overline{Q}_I(v_{\gamma})$ are specified
 by $I = (i_1,i_2,\cdots, i_t)$ with $$0 < i_1 \leq i_2 \leq
\cdots \leq i_t < \infty, $$ and $$\overline{Q}_I(v_{\gamma})=
\overline{Q}_{i_i} \overline{Q}_{i_2}\cdots
\overline{Q}_{i_t}(v_{\gamma}).$$

\item The formula $$\overline{Q}_I(\alpha x) = {\alpha}^{2^t}\overline{Q}_I(x)$$ holds for
$\alpha \in k$ and $I = (i_1,i_2,\cdots, i_t)$ by Theorem
\ref{thm:araki.kudo.dyer.lashof.operations.over.overline.F}.

\item The empty sequence $I=\emptyset$ is allowed and in this case
$\overline{Q}_{\emptyset}(v_{\gamma})$ is defined to be
$v_{\gamma}$.
\end{enumerate}

Define the weight of a monomial $\overline{Q}_{I}(v_{\gamma})$ to be
$w(\overline{Q}_{I}(v_{\gamma})) = 2^t$ where $I = (i_1,i_2,\cdots,
i_t)$. The weight of a product is defined by $w(X\cdot Y ) = w(X)
+w(Y).$

A basis for the symmetric algebra $\HH_\bu(Q(X),\overline{\F}_2)$,
isomorphic to $S[\overline{Q}_I(v_{\gamma})]$, is given by a choice
of products of monomials for a polynomial algebra
$$\mathcal B(V)=\{\overline{Q}_{I_1}(v_{\gamma_1}) \cdots
\overline{Q}_{I_k}(v_{\gamma_k})\}.$$ This feature follows with
coefficients in $\overline{\F}_2$ directly from the above remarks
together with the universal coefficient theorem and the analogous
results in \cite{clm} where coefficients are taken in $\F_2$.

Then a basis for $\HH_\bu(\Sigma_d, V^{\otimes
d}\otimes_{\F_2}\overline{\F}_2)$ are those monomials in $\mathcal
B(V)$ with $w(\overline{Q}_{I_1}(v_{\gamma_1})) + \cdots +
w(\overline{Q}_{I_k}(v_{\gamma_k})) = d.$ The action of $GL_{n}(k)$
on $\oplus_{d \geq 0} \HH_\bu(\Sigma_d, V^{\otimes d})$ follows from
the formula
$$\overline{Q}_I(\alpha \cdot x + \beta \cdot y) = \alpha^{2^t}\cdot \overline{Q}_I(x) + \beta^{2^t}\cdot
\overline{Q}_I(y)$$ for scalars $\alpha$ and $\beta$ with $I =
(i_1,i_2,\cdots, i_t)$ for $0 < i_1 \leq i_2 \leq \cdots \leq i_t$.
These elements also have a degree which correspond to the natural
degrees in $\HH_\bu(\Sigma_d, V^{\otimes d})$ (cf. Section
\ref{sec:degree.shift}).

\section{Odd primes}\label{sec:Odd.primes}
\subsection{}
We now describe the known homology groups $\HH_\bu(Q(X),\F_p)$ and
then  describe the $GL(V)$-action on $\HH_\bu(Q(X),k)$ for $p$ odd.
The methods are similar to those in Section \ref{sec:The.prime.two}.

Recall the homology of $Q(X)$ over $\F_p$ for odd primes $p$. There
are operations $$Q_j:\HH_n(Q(X),\F_p) \to
\HH_{j(p-1)+pn}(Q(X),\F_p)$$ for which

\begin{itemize}
  \item[(1)] $n+j \equiv 0(\text{mod }2)$ and
  \item[(2)] $j > 0$.
  \item[(3)] That is, each operation $Q_j$ is defined on even dimensional
  classes in case $j>0$ is even and  on odd dimensional classes in case $j$ is odd.
\end{itemize}

There is an additional operation given by the Bockstein
$$\beta:\HH_n(Q(X),\F_p)\to \HH_{n-1}(Q(X),\F_p)$$ for which $\beta^0$ denotes the
identity map $$\beta^0:\HH_n(Q(X),\F_p) \to \HH_n(Q(X),\F_p).$$

Let $\{v_{\alpha}\}$ denote a choice of basis for the reduced
homology groups $V = \overline{\HH}_\bu(X,\F_p)$. Then there is an
isomorphism of algebras
$$\HH_\bu(Q(X),\F_p) \to S[Q_J(v_{\alpha})|\hbox{deg}(Q_J(v_{\alpha}))\equiv 0(\text{mod }2)]
\otimes E[Q_J(v_{\alpha})|\hbox{deg}(Q_J(v_{\alpha}))\equiv
1(\text{mod }2)]$$ where
\begin{itemize}
\item[(4)] $\hbox{deg}(x)$ denotes the degree of an element $x$,
\item[(5)]  $S[Q_J(v_{\alpha})]$ denotes the polynomial algebra with
  generators $Q_J(v_{\alpha})$  (which may
  contain Bocksteins as given in the next paragraph), and
\item[(6)] $E[Q_J(v_{\alpha})]$ denotes the exterior algebra with
  generators $Q_J(v_{\alpha})$ (which may
  contain Bocksteins as given in the next paragraph).
\end{itemize}

The elements $Q_J(v_{\alpha})$ are described next where $J =
(\epsilon_1,j_1,\epsilon_2,j_2,\cdots, \epsilon_t,j_t)$ with $0 <
j_1 \leq j_2 \leq \cdots \leq j_t < \infty, $ $\epsilon_s = 0,1$
with $1 \leq s \leq t$. Then $$Q_J(v_{\alpha}) =
\beta^{\epsilon_1}Q_{j_1}\beta^{\epsilon_2}Q_{j_2}\cdots,
  \beta^{\epsilon_t}Q_{j_t}(v_{\alpha})$$ {\it whenever that operation is defined}.

The monomials in $Q_J(v_{\alpha})$ can be interpreted in terms of
the homology of the symmetric groups with coefficients in the tensor
powers of $\overline \HH_\bu(X,\F_p)$ as follows. Each monomial
$Q_J(v_{\alpha})$ has a weight given by $w(Q_J(v_{\alpha}))= p^t$
for $J = (\epsilon_1,j_1,\epsilon_2,j_2,\cdots, \epsilon_t,j_t).$
Furthermore, the weight of a product is defined by $w(X\cdot Y ) =
w(X) +w(Y).$ In case of the empty sequence $J=\emptyset$,
$Q_\emptyset(v_{\alpha})$ is defined to be $v_{\alpha}$.

Then the group $\HH_\bu(\Sigma_d, V^{\otimes d})$ for $V = \overline
\HH_\bu(X, \F_p)$ is the linear span of the product of monomials in
$$S[Q_J(v_{\alpha})|\hbox{deg}(Q_J(v_{\alpha}))\equiv 0(\text{mod }2)] \otimes
E[Q_J(v_{\alpha})|\hbox{deg}(Q_J(v_{\alpha}))\equiv 1(\text{mod
}2)]$$ of weight exactly $d$. As above, these elements also have a
degree which correspond to the natural degrees in $\HH_\bu(\Sigma_d,
V^{\otimes d})$.

\subsection{} The natural modifications for coefficients in $\overline{\F}_p$ are
stated next. In this section, homology is taken with
$\overline{\F}_p$-coefficients for odd primes $p$. There are
operations
$$\overline{Q}_j:\HH_n(Q(X),\overline{\F}_p) \to \HH_{j(p-1)+pn}(Q(X),\overline{\F}_p)$$ for
which

\begin{itemize}
  \item[(1)] $n+j \equiv 0(\text{mod }2)$ and
  \item[(2)] $j > 0$.
  \item[(3)] That is, each operation $\overline{Q_j}$ is defined (i) on even dimensional
  classes in case $j$ is even with $j>0$ and (ii) on odd dimensional classes in case $j$ is odd.
\end{itemize}

There is an additional operation given by the Bockstein
$$\beta:\HH_n(Q(X),\overline{\F}_p)\to \HH_{n-1}(Q(X),\overline{\F}_p)$$ for which $\beta^0$ denotes the
identity map $$\beta^0:\HH_n(Q(X),\overline{\F}_p) \to
\HH_n(Q(X),\overline{\F}_p).$$

Let $\{v_{\alpha}\}$ denote a choice of basis for the reduced
homology groups $V = \overline{\HH}_\bu(X,\overline{\F}_p)$. Then
there is an isomorphism of algebras
$$\HH_\bu(Q(X),\overline{\F}_p) \to S[\overline{Q}_J(v_{\alpha})|
\hbox{deg}(\overline{Q}_J(v_{\alpha}))\equiv 0(\text{mod }2)]
\otimes
E[\overline{Q}_J(v_{\alpha})|\hbox{deg}(\overline{Q}_J(v_{\alpha}))\equiv
1(\text{mod }2)]$$ where
\begin{itemize}
\item[(4)] $\hbox{deg}(x)$ denotes the degree of an element $x$,
\item[(5)]  $S[\overline{Q}_J(v_{\alpha})]$ denotes the polynomial algebra with
  generators $\overline{Q}_J(v_{\alpha})$ (which may
  contain Bocksteins), and
\item[(6)] $E[\overline{Q}_J(v_{\alpha})]$ denotes the exterior algebra with
  generators $\overline{Q}_J(v_{\alpha})$ (which may
  contain Bocksteins).
\end{itemize}

The elements $\overline{Q}_J(v_{\alpha})$ are described next where
$J = (\epsilon_1,j_1,\epsilon_2,j_2,\cdots, \epsilon_t,j_t)$ with $0
< j_1 \leq j_2 \leq \cdots \leq j_t < \infty, $ $\epsilon_s = 0,1$
with $1 \leq s \leq t$. Then $$\overline{Q}_J(v_{\alpha}) =
\beta^{\epsilon_1}\overline{Q}_{j_1}\beta^{\epsilon_2}\overline{Q}_{j_2}\cdots,
  \beta^{\epsilon_t}\overline{Q}_{j_t}(v_{\alpha})$$ {\it whenever that operation is defined}.

The monomials in $\overline{Q}_J(v_{\alpha})$ can be interpreted in
terms of the homology of the symmetric groups with coefficients in
the tensor powers of $\overline \HH_\bu(X,\overline{\F}_p)$ as
follows. Each monomial $\overline{Q}_J(v_{\alpha})$ has a weight
given by $w(\overline{Q}_J(v_{\alpha}))= p^t$ for $J =
(\epsilon_1,j_1,\epsilon_2,j_2,\cdots, \epsilon_t,j_t).$
Furthermore, the weight of a product is defined by $w(X\cdot Y ) =
w(X) +w(Y).$ In case of the empty sequence $J=\emptyset$,
$\overline{Q}_\emptyset(v_{\alpha})$ is defined to be $v_{\alpha}$.

Then the group $\HH_\bu(\Sigma_d, V^{\otimes d})$ for $V = \overline
\HH_\bu(X,\overline{\F}_p)$ is the linear span of the product of
monomials in
$$S[\overline{Q}_J(v_{\alpha})|d(\overline{Q}_J(v_{\alpha}))\equiv 0
(\text{mod }2)] \otimes
E[\overline{Q}_J(v_{\alpha})|d(\overline{Q}_J(v_{\alpha}))\equiv 1
(\text{mod }2)]$$ of weight exactly $d$.

Furthermore $\overline{Q}_i(\lambda x) = (\lambda)^p
\overline{Q}_i(x)$ for $\lambda \in \overline{\F}_p$. Thus
$$\overline{Q}_J(\lambda v_{\alpha}) =
{\lambda}^{p^t}\overline{Q}_J(v_{\alpha})$$ for $J =
(\epsilon_1,j_1,\epsilon_2,j_2,\cdots, \epsilon_t,j_t).$ This
formula follows by iterating the formula $$\overline{Q}_i(\lambda x)
= {\lambda}^{p}\overline{Q}_i(x)$$ as stated in Theorem
\ref{thm:araki.kudo.dyer.lashof.operations.over.overline.F}.

\section{Homology operations for $Q(X)$ over $\overline{\F}_p$} \label{sec:Homology operations for $QX$ over $F$, 10.july.2007}
\subsection{}

In this section we work out the formula $\overline{Q}_i(\lambda x) =
\lambda^p \overline{Q}_i(x)$ for $\lambda \in \overline{\F}_p$ as
stated in Lemma \ref{lem:scalars} below. The proof of this lemma
then finishes the proof of Theorem
\ref{thm:araki.kudo.dyer.lashof.operations.over.overline.F}.

The operations $Q_i$, originally due to Araki and Kudo for $p = 2$
with odd primary versions due to Dyer and Lashof, were defined over
the field with $p$ elements \cite{ak,dl}. These operations in the
special case for the homology of the symmetric groups were implicit
in Nakaoka's computations \cite{Nakaoka}. In addition, these
operations admit extensions to homology taken with coefficients in
$\overline{\F}_p$ as described above via the universal coefficient
theorem. Properties of these operations over $\overline{\F}_p$ are
obtained from the structure of $\HH_\bu(\Sigma_p, V^{\otimes p})$.
These homology groups are easy to work out using the $p$-Sylow
subgroup of $\Sigma_p$.

As preparation, it is convenient to first recall properties of
$V^{\otimes p}$ where $V$ is a vector space over a field $k$ of
characteristic $p$ with further choices of either $k = \F_p$, or $k
= \overline{\F}_p$ made explicit below. The cyclic group of order
$p$ generated by the $p$-cycle $\sigma_p = (1,2, \cdots, p)$ acts
naturally on $V^{\otimes p}$. Thus $V^{\otimes p}$ is naturally a
$k[\Z/p\Z]$-module.

Choose a totally ordered basis for $V$, say $e_{\alpha}$ for $\alpha
\in S$. A basis for $V^{\otimes p}$ is described next. Consider the
multi-index  $A = (\alpha_1, \cdots \alpha_p)$, $\alpha_i \in S$,
with $$e_{A} = e_{\alpha_1}\otimes \cdots \otimes e_{\alpha_p}.$$

\begin{itemize}
\item[(1)] Given the set $S$, consider the $p$-fold product $S^{\times p}$.
Observe that  $\Z/p\Z$ acts on the set $S^{\times p}$ via the
$p$-cycle $\sigma_p = (1,2, \cdots, p)$.

\item[(2)] Let $\Delta^{\times p}(S)$ denote the diagonal subset with
$$\Delta^{\times p}(S)= \{(\alpha_1, \cdots \alpha_p)\in
S^{\times p} | \alpha_i = \alpha_j \hbox{ for all } i, j\}.$$

\item[(3)] Let $\Gamma^{\times p}(S)$ denote the complement of
$\Delta^{\times p}(S)$ in $S^{\times p}$. Thus $$\Gamma^{\times
p}(S) = \{(\alpha_1, \cdots \alpha_p)\in S^{\times p} | \alpha_i\neq
\alpha_j \hbox{ for some } i < j\}.$$

\item[(4)] The action of $\Z/p\Z$ on the set $S^{\times p}$ restricts to
actions on both $\Delta^{\times p}(S)$, and $\Gamma^{\times p}(S)$.
The set $S^{ \times p}$ is a disjoint union of sets
$$S^{\times p} = \Delta^{\times p}(S) \amalg \Gamma^{\times p}(S)$$
as a $\Z/p\Z$-set.

\item[(5)] Define the set $T(S)$ by a choice of elements, one in each
$\Z/p\Z$-orbit in $\Gamma^{\times p}(S)$.

\item[(6)] In case $S$ is finite of cardinality $N$, then
\begin{enumerate}
\item $S^{ \times p}$ has cardinality $N^p$,
\item $\Delta^{\times p}(S)$ has cardinality $N$,
\item $\Gamma^{\times p}(S)$ has cardinality $N^p -N$
which is divisible by $p$, and
\item the cardinality of $T(S)$ is $(1/p)(N^p-N)$.
\end{enumerate}
\end{itemize}

A specific choice of basis for $V^{\otimes p}$ consists to two types
of elements as follows.
\begin{itemize}
\item[(1)] Define $e_A = e_{\alpha}^{\otimes p} = \overbrace{e_{\alpha}\otimes \cdots \otimes e_{\alpha}}^{\text{$p$
times}}$, with $A = (\alpha, \cdots \alpha) \in \Delta^{\times
p}(S)$ for all $\alpha \in S$.
\item[(2)] Define $e_B = e_{\alpha_1} \otimes \cdots \otimes e_{\alpha_p}$, with $B = (\alpha_1, \cdots, \alpha_p) \in T(S)$
where at least two of the $\alpha_j$ differ for all $B \in T(S)$.
\end{itemize}

Consider the $k[\Z/p\Z]$-modules given by
\begin{itemize}
\item[(1)] the cyclic $k[\Z/p\Z]$-module spanned by $e_{A}= \overbrace{e_{\alpha} \otimes \cdots \otimes e_{\alpha}}^{\text{$p$
times}}$ denoted $<e_A>$ for $A \in \Delta^{\times p}(S)$, and
\item[(2)] the cyclic $k[\Z/p\Z]$-module spanned by a \emph{choice} of the elements $e_{B} = e_{\alpha_1}
\otimes \cdots \otimes e_{\alpha_p}$ denoted $<e_B>$ with $B =
(\alpha_1, \cdots, \alpha_p) \in T(S)$.
\end{itemize}

The choices of basis elements above then give a direct sum
decomposition of  $V^{\otimes p}$.

\begin{enumerate}
\item Let $${V^{\otimes p}}(\mbox{fix})$$ denote the linear span in $V^{\otimes p}$ of the $<e_A>$ for all $A \in
\Delta^{\times p}(S)$.

\item Let $$V^{\otimes p}(\mbox{free})$$ denote the linear span in $V^{\otimes p}$ of the $<e_B>$ for all $B =
(\alpha_1, \cdots, \alpha_p) \in T(S)$.
\end{enumerate}

Versions of the next lemma arose in work of P.~A.~Smith, and
Steenrod in which $V^{\otimes p}(\mbox{free})_{\Z/p\Z}$ denotes the
natural module of coinvariants.
\begin{lem} \label{lem:direct.sum.decomposition}
Let $k$ be any field of characteristic $p$. The natural inclusions
$${V^{\otimes p}}(\mbox{fix}) \oplus V^{\otimes p}(\mbox{free})\to
V^{\otimes p}$$ give an isomorphism of $k[\Z/p\Z]$-modules. Thus
there are induced isomorphisms
$$\HH_\bu(\Z/p\Z, {V^{\otimes p}}(\mbox{fix}) \oplus
V^{\otimes p}(\mbox{free})) \to \HH_\bu(\Z/p\Z, {V^{\otimes p}}),$$
and $$\HH_\bu(\Z/p\Z, {V^{\otimes p}}(\mbox{fix})) \oplus
\HH_\bu(\Z/p\Z,V^{\otimes p}(\mbox{free})) \to \HH_\bu(\Z/p\Z,
{V^{\otimes p}}).$$

Furthermore,
\begin{enumerate}
  \item ${V^{\otimes p}}(\mbox{fix})$ is a trivial
 $k[\Z/p\Z]$-module, and
  \item $V^{\otimes p}(\mbox{free})$ is a free
$k[\Z/p\Z]$-module.
\end{enumerate}

Thus there are isomorphisms $$(\HH_\bu(\Z/p\Z,k) \otimes_k
{V^{\otimes p}}(\mbox{fix})) \oplus V^{\otimes
p}(\mbox{free})_{\Z/p\Z} \to \HH_\bu(\Z/p\Z, {V^{\otimes p}}).$$
\end{lem}

This information is used to prove the next lemma.

\begin{lem} \label{lem:scalars}
Assume that $k = \overline{\F}_p$ with $x \in V =
\HH_\bu(X,\overline{\F}_p)$. The formula $\overline{Q}_i(\lambda x)
= \lambda^p \overline{Q}_i(x)$ for $\lambda \in \overline{\F}_p$ is
satisfied.
\end{lem}

\begin{proof} Assume that all homology groups are taken with coefficients
in $k = \overline{\F}_p$. Observe that the natural map
$$\HH_\bu(\Z/p\Z, V^{\otimes p}) \to \HH_\bu(\Sigma_p, V^{\otimes p})$$
is an epimorphism as $\Z/p\Z$ is the $p$-Sylow subgroup of
$\Sigma_p$ and $\overline{\F}_p$ has characteristic $p$.
Furthermore, $$\HH_\bu(\Z/p\Z, V^{\otimes p})$$ is described
classically (by Smith and by Steenrod) as in Lemma
\ref{lem:direct.sum.decomposition} obtained from isomorphisms of
$\F[\Z/p\Z]$-modules
$$V^{\otimes p} \to {V^{\otimes p}}(\mbox{fix}) \oplus V^{\otimes p}(\mbox{free}).$$

Let $B_\bu$ denote the classical minimal free resolution of $k$
regarded as a trivial $k[\Z/p\Z]$-module with basis for $B_i =
k[\Z/p\Z]$ given by $g_i$ in degree $i$. Next, consider the homology
of the chain complex  $$B_\bu \otimes_{k[\Z/p\Z]} <e_A>$$ for $A =
(\alpha, \cdots \alpha) \in \Delta^{\times p}(S)$. Thus $$e_A =
\overbrace{e_{\alpha}\otimes \cdots \otimes e_{\alpha}}^{\text{$p$
times}} = {e_{\alpha}}^{\otimes p}.$$

Furthermore, the class of $g_i \otimes (\lambda\cdot
e_{\alpha})^{\otimes p}$ is equal to the class of $\lambda^p g_i
\otimes (e_{\alpha})^{\otimes p}$ in the homology of $B_\bu
\otimes_{k[\Z/p\Z]} <e_A>$. Since the element
$\overline{Q}_i(\lambda x)$ is identified with the image of $g_i
\otimes (\lambda\cdot e_{\alpha})^{\otimes p}$ up to a scalar
multiple depending on $i$ and the degree of $x$ in the homology of
$Q(X)$, the lemma follows.
\end{proof}

Some remarks about graded vector spaces $V$ and the sign
representation are given in the next section.

\begin{rem}\label{remark:sign.rep}
The cyclic $k[\Z/p\Z]$-modules spanned by $e_{A}=
\overbrace{e_{\alpha} \otimes \cdots \otimes e_{\alpha}}^{\text{$p$
times}}$ denoted $<e_A>$ are all trivial $k[\Z/p\Z]$-modules where
each such module is concentrated in degree $p \cdot
\hbox{deg}(e_{\alpha})$, where $\hbox{deg}(e_{\alpha})$ denotes the
degree of the element $e_{\alpha}$.

The modules $<e_A>$ are not always trivial $k\Sigma_p$-modules as
the action depends on the degree of $e_{A}$. In particular, there
are isomorphisms of $k\Sigma_p$-modules given by
\[
 e_A =
\begin{cases}
k & \text{if $\hbox{deg}(e_{\alpha})$ is even, and}\\
\hbox{sgn} & \text{if $\hbox{deg}(e_{\alpha})$ is odd.}
\end{cases}
\]

The structure of the operations $Q_i(-)$ in case $k = \F_p$, and
$\overline{Q}_i(-)$ in case $k = \overline{\F}_p$ `record' this
structure, and thus give information about coefficients in the sign
representation. This trick has been used extensively in working out
the homology of certain mapping class groups with coefficients in
$\hbox{sgn}$.
\end{rem}

\section{Topological analogues}\label{sec:Topological analogues}
\subsection{}

The purpose of this section is to explain the relationship in more
detail between the homology of $Q(X)$ and the groups $ \oplus_{d
\geq 0}\HH_\bu(\Sigma_d, V^{\otimes d})$ where homology is taken
with coefficients in any field $k$. First, recall that if $V$ is a
connected graded vector space over $k$, then a choice of space $X$
was given in Section \ref{sec: On the homology of $Q(X)$} which has
reduced homology given by $V$: one choice of $X$ is a wedge of
spheres with reduced homology given by $V$. This choice of $X$ as a
wedge of spheres has a second feature which is addressed next.

Namely, the reduced homology groups of a wedge of spheres is a free
abelian group. Consider the case for which $V$ is obtained by
tensoring a free module $M$ over the integers $\Z$ with the field
$k$, namely $$ V = M \otimes_{\Z}k.$$ In this case $X$ may be chosen
to have the analogous property given by $M = \bar{H}_*(X,\Z)$ and
$$V = \bar{H}_*(X,\Z) \otimes_{\Z}k.$$

The point of view here is that, with mild conditions concerning the
graded vector space $V$, there is a topological space $X$ with the
property that the homology of the space $Q(X)$, a functor of $X$,
has homology given by $$ \oplus_{d \geq 0}\HH_\bu(\Sigma_d,
V^{\otimes d}).$$ This feature then provides a transparent way to
determine the $GL(V)$-action on $ \oplus_{d \geq 0}\HH_\bu(\Sigma_d,
V^{\otimes d}).$

To describe this isomorphism in more detail, it is useful from the
topological point of view to have base-points. Thus this section
will be restricted to path-connected, pointed CW-complexes $(X,*)$
where $*$ is the base-point of $X$. Maps are required to preserve
base-points denoted $*$. Furthermore, a wedge of pointed spaces has
a natural base-point. Thus the example of a wedge of spheres given
in Section \ref{sec: On the homology of $Q(X)$} suffices.

Next let $E\Sigma_d$ denote a contractible Hausdorff space which has
a free (right) action of $\Sigma_d$. Let $X^d$ denote the $d$-fold
product and $X^{(d)}$ the $d$-fold smash product given by $X^{(d)} =
X^d/S(X^d)$ where $S(X^d)$ denotes the subspace of $X^d$ given by
$$S(X^d) = \{(x_1, \cdots,x_d)| x_i = * \hbox{ for some} \ 1 \leq i
\leq d\}.$$ In addition, let $*$ also denote the class of the
base-point in $X^{(d)}$.

Consider the functor from ``pointed spaces" to ``pointed spaces"
which sends the space $X$ to the space given by $$(E\Sigma_d
\times_{\Sigma_d} X^{(d)})/(E\Sigma_d \times_{\Sigma_d}\{*\}).$$
This last construction is denoted $D_d(X)$ in what follows.

Some properties of the space $D_d(X)$ are listed next where it is
assumed that $X$ is a path-connected CW-complex with $V =
\overline{\HH}_\bu(X,k),$ the reduced homology groups of $X$ over
the field $k$.

\begin{enumerate}
\item The construction $D_d(X)$ gives a functor $D_d(-)$ with values
$D_d(X)$ for any pointed space $X$. Thus $D_d(-)$ is a functor from
pointed spaces to pointed spaces.
  \item The reduced (singular) homology of $X^{(d)}$ with any
  field coefficients $k$ is isomorphic to $V^{\otimes
  d}$ as a (left) $\Sigma_d$-module.
  \item The reduced homology of $D_d(X)$ is isomorphic to $\HH_\bu(\Sigma_d, V^{\otimes
  d}).$

\item Consider the functor from pointed spaces to pointed spaces
which sends the space $X$ to the space given by $Q(X)$. Then $Q(X)$
satisfies the property that there is an isomorphism of algebras
$$\HH_\bu(Q(X),\F) \to \oplus_{d \geq 0}\HH_\bu(\Sigma_d, V^{\otimes d}).$$
\end{enumerate}

\subsection{} A theorem originally due to Kahn \cite{k}, subsequently proven
by Snaith \cite{snaith} with extensions to related configuration
spaces as well as variations in \cite{cmt} can be stated as follows.

\begin{thm}\label{thm:astable.splitting}
Let $X$ denote a path-connected CW-complex. Then there is a natural
stable homotopy equivalence $$H\colon Q(X) \longrightarrow
\bigvee_{d \geq 0} D_d(X).$$

Thus there are isomorphisms  in homology, natural for pointed spaces
$(X,*)$, $$E_*(Q(X))\to E_*(\bigvee_{d \geq 0} D_d(X))$$ for any
homology theory $E_*(-)$. In particular, if $V = \overline
{\HH}_\bu(X,k) $ is regarded as a graded vector space over a field
$k$, there are isomorphisms $$\HH_\bu(Q(X),\F)\to \oplus_{d \geq
0}\HH_\bu(\Sigma_d, V^{\otimes d}).$$

\end{thm}

\subsection{}Some remarks about where these results can be found are
listed next. An efficient computation of the homology of $Q(X)$ is
given in the proof of Theorem 4.2 on pages 40-47 of \cite{clm}. The
stable decomposition in Theorem \ref{thm:astable.splitting} is
proven in \cite{k,snaith}. A short two page
 proof is in the appendix of \cite{c}.

That the algebraic weights for the homology of $Q(X)$ agree with
those arising from the stable decompositions follows from the way in
which the operations are defined. One sketch is in \cite{clm}, pages
237-243.

Finally, a remark about notation: there are two different notations
for homology operations in use. The `lower notation' $Q_i(-)$ is
used above while `upper notation' $Q^s(-)$ is used in \cite{clm}.
The translation is as follows:
\begin{enumerate}
  \item $p = 2$: $Q_{s-q}(x) =Q^{s}(x)$ if $ s > q = \hbox{degree}(x)$.
  \item $p > 2$: $Q_{(2s-q)}(x) = c Q^s(x)$ for a nonzero scalar $c$ if $ 2s > q= \hbox{degree}(x)$,
  a choice of notation slightly different than that in \cite{clm}, foot of page 7.
\end{enumerate}

\section{A degree shift}\label{sec:degree.shift}

\subsection{}
In the work above in which the $GL(V)$-action on $\oplus_{d \geq
0}\HH_\bu(\Sigma_d, V^{\otimes d})$ was analyzed, it was
specifically assumed that $V$ is the reduced homology of a
path-connected topological space $X$, namely
$$V  = \bar{\HH}_\bu(X,k).$$ However, the reduced homology of any
path-connected space is concentrated in degrees greater than $0$.

Thus to address the structure of $\oplus_{d \geq 0}\HH_\bu(\Sigma_d,
W^{\otimes d})$ where $W$ is a graded vector space concentrated in
degree $0$, some technical modifications are required. This
modification is achieved through a formal degree shift which is
addressed in this section. One way to achieve this modification is
as follows.

\begin{defin} \label{defin:degree.shift}
Let $W$ be a vector space over $k$ concentrated in degree $0$. Given
a fixed natural number $n \in \mathbb N$ define a graded vector
space over $k$ denoted $$(n,W)$$ which is
\begin{enumerate}
  \item concentrated in degree $n$, namely $$(n,W) = \{0\}$$ in
  degrees not equal to $n$, and
  \item $(n,W)$ in degree $n$ is isomorphic to $W$.
\end{enumerate}
Thus there is a morphism of vector spaces (which does not respect
gradation) given by $$\sigma_n: W \to (n,W)$$ where
$$\sigma_n(x) = (n,x).$$
\end{defin}

Elementary features of $\sigma_n: W \to (n,W)$ are stated next.

\begin{lem} \label{lem:degree.shift}
The morphism of vector spaces (which does not respect gradation)
given by $$\sigma_n: W \to (n,W)$$ is an isomorphism of underlying
vector spaces. Furthermore, if the rank of $W$ is at least one, the
induced map given by the $d$-fold tensor product, $ d \geq 2$, of
$\sigma_n$,
$$(\sigma_n)^{\otimes d}: W^{\otimes d} \to (n,W)^{\otimes d}$$ is
an isomorphism of underlying modules over $k\Sigma_d$ provided $n$
is even.
\end{lem}

\begin{rem} \label{rem:degree.shift}
If $n$ is odd, and $d \geq 2$, then the map of vector spaces
$$(\sigma_n)^{\otimes d}: W^{\otimes d} \to
(n,W)^{\otimes d}$$ is not an isomorphism of underlying modules over
$k\Sigma_d$. In particular if $n$ is odd and $W$ is of rank one,
then
\begin{enumerate}
\item $W^{\otimes 2}$ is a vector space of rank one which is the
trivial representation of $\Sigma_2$, and
\item $(n,W)^{\otimes 2}$ is still a vector space of rank one, but is the
sign representation of $\Sigma_2$.
\end{enumerate}

Furthermore, if $n$ is even, the map
$$(\sigma_n)^{\otimes d}: W^{\otimes d} \to
(n,W)^{\otimes d}$$ is both an isomorphism of vector spaces and of
modules over $k\Sigma_d$. The most economical way to achieve this
is by setting $n = 2 $ as is done below.
\end{rem}

Definition \ref{defin:degree.shift} provides a natural degree shift
for the isomorphism given in Sections \ref{sec:The.prime.two} and
\ref{sec:Odd.primes} above. This will ensure the degrees match with
the natural degrees in $\Hs Y^\lambda)$. Namely, start with a graded
vector space $W$ concentrated in degree $0$.

Then $$(\sigma_{2n})^{\otimes d}: W^{\otimes d} \to (2n,W)^{\otimes
d}$$ induces an isomorphism of underlying $\Sigma_d$-modules ( which
does not preserve degrees ). Thus there is an induced natural shift
map on the level of homology which gives isomorphisms
$$\Theta_d(2n):\HH_s(\Sigma_d, W^{\otimes d}) \to
\HH_{s+2nd}(\Sigma_d, (2n,W)^{\otimes d})$$ as well as isomorphisms
$$\oplus_{d \geq 0}\Theta_d(2n): \oplus_{d \geq 0}\HH_\bu(\Sigma_d, W^{\otimes d}) \to \oplus_{d \geq 0} \HH_{\bu+2nd}(\Sigma_d,
(2n,W)^{\otimes d}).$$

This all is described by the following theorem, which uses Lemma
\ref{lem:scalars} with $k = \overline{\F}_p$.

\begin{thm}\label{thm:shiftedfinalanswer for p=2 and p oddL}

Let $W$ be a graded vector space concentrated in degree $0$ with
basis $\{w_{\gamma}|\gamma \in S\}$ over $\overline{\F}_p$. Then
$$\oplus_{d \geq 0}\HH_\bu(\Sigma_d, W^{\otimes d})$$ is a
graded,commutative algebra given as follows.

\begin{enumerate}
  \item[(a)] If $ p = 2$, then $\oplus_{d \geq 0}\HH_\bu(\Sigma_d, W^{\otimes d})$
is isomorphic to the polynomial algebra
$S[\overline{Q}_I(w_{\gamma})]$ with generators
$\overline{Q}_I(w_{\gamma})$, $I = (i_1,i_2,\cdots, i_t)$, $0 < i_1
\leq i_2 \leq \cdots \leq i_t < \infty$ where the degree of
$\overline{Q}_I(w_{\gamma})$ is $$i_1 + 2i_2 + 4i_3 + \cdots +
2^{i_t - 1}i_t.$$

\item[(b)] If $ p > 2$, then $\oplus_{d \geq 0}\HH_\bu(\Sigma_d, W^{\otimes d})$
is isomorphic to the symmetric algebra $$
S[\overline{Q}_J(w_{\gamma})|\hbox{deg}(\overline{Q}_J(w_{\gamma}))\equiv
0(\text{mod }2)] \otimes
E[\overline{Q}_J(w_{\gamma})|\hbox{deg}((\overline{Q}_J(w_{\gamma}))\equiv
1(\text{mod }2)]$$ for which $J =
(\epsilon_1,j_1,\epsilon_2,j_2,\cdots, \epsilon_t,j_t)$ with $0 <
j_1 \leq j_2 \leq \cdots \leq j_t < \infty, $ $\epsilon_k = 0,1$
with $1 \leq k \leq t$,

    \begin{enumerate}
    \item[(i)]  $S[\overline{Q}_J(w_{\gamma})]$ denotes the polynomial algebra with
    generators $\overline{Q}_J(w_{\gamma})$ (which may
    contain Bocksteins but must start with $\overline{Q}_{j_t}(w_{\gamma})$ and is not allowed
    to start with $\beta(w_{\gamma})$),
    \item[(ii)] $E[\overline{Q}_J(w_{\gamma})]$ denotes the exterior algebra with
    generators $\overline{Q}_J(w_{\gamma})$ ( which may
    contain Bocksteins but must start with $\overline{Q}_{j_t}(w_{\gamma})$ and is not allowed
    to start with $\beta(w_{\gamma}))$, and
    \item[(iii)] the degree of $\overline{Q}_J(w_{\gamma})$ for $J = (\epsilon_1,j_1,\epsilon_2,j_2,\cdots,
\epsilon_t,j_t)$  is equal to
$$(-\epsilon_1 + j_1(p-1))+p(-\epsilon_2 + j_2(p-1))) +  \cdots + p^{t-1}(-\epsilon_t +
j_t(p-1))).$$
    \end{enumerate}

\end{enumerate}

The action of $GL(W)$ on
$$\bigoplus_{d \geq 0}\HH_\bu(\Sigma_d, W^{\otimes
d})$$ is given in terms of the formulas in Section
\ref{sec:The.prime.two} in case $p = 2$ and Section
\ref{sec:Odd.primes} in case $p$ is odd for which
$$\overline{Q}_i(\lambda x) = \lambda^p \overline{Q}_i(x)$$ with $\lambda
\in \overline{\F}_p$.

\end{thm}

\subsection{} Let $k = \overline{\F}_p$. The description of the action of $GL(W)$
from Sections \ref{sec:The.prime.two} and \ref{sec:Odd.primes},
shows that the action on monomials of the form $\bar{Q}_i(w)$ or
$\bar{Q}_{\epsilon_1,i_1}(w)$ is through a Frobenius twist. That is,
this action gives $GL(W)$-modules which are isomorphic to $W^{(1)}$.
A monomial $\overline{Q}_{i_1,i_2}(w)$ would correspond to the
$GL(W)$-module $W^{(2)}$. Monomials
$\overline{Q}_i(w)\overline{Q}_i(w)$ would give a $S^2(W^{(1)})$. In
odd characteristic, the  square of a monomial in the exterior
algebra part would give zero, so we obtain modules of the form
$\Lambda^c(W^{(c)})$ but no symmetric powers of such modules. Thus
Theorem \ref{thm:shiftedfinalanswer for p=2 and p oddL} implies the
following.
\begin{cor}\label{corgivingstructure} Let $k = \overline{\F}_p$, and let $W$
denote a vector space over $k$ of rank $n$
concentrated in degree zero, so $GL(W)=GL_n(k)$. In characteristic
two, the $GL_n(k)$-module $\HH_\bu(\Sigma_d, W^{\otimes d})$ is a
direct sum of modules of the form:

\begin{equation}
\label{eq: descriptionofsummandsinchar2}
S^{a_1}(W) \otimes S^{a_2}(W^{(c_2)}) \otimes \cdots \otimes S^{a_s}(W^{(c_s)})
\end{equation}
where each $a_i \geq 0$, $c_i>0$ and  $d=a_1 + \sum_{j=2}^s a_j2^{c_j}.$

In odd characteristic the $GL_n(k)$-module $\HH_\bu(\Sigma_d,
W^{\otimes d})$ is a direct sum of modules of the form
\begin{equation}
\label{eq: descriptionofsummandsincharodd}
S^{a_1}(W) \otimes S^{a_2}(W^{(c_2)}) \otimes \cdots \otimes S^{a_s}(W^{(c_s)})\otimes \Lambda^{d_2}(W^{(d_2)})
\otimes \cdots \otimes \Lambda^{d_t}(W^{(d_t)})
\end{equation} where each $a_i
\geq 0$, each $c_i, d_i>0$ and where
$d=a_1 + \sum_{j=2}^s a_jp^{c_j}+ \sum_{j=2}^td_jp^{d_j}.$
\end{cor}

\begin{rem}
\label{rem: cor821remark} Corollary \ref{corgivingstructure} cannot be used alone to
determine $\HH_i(\Sigma_d, W^{\otimes d})$ as a $GL(W)$-module, additional details on degrees of
monomials in Theorem \ref{thm:shiftedfinalanswer for p=2 and p oddL} are needed. To compute $\HH_i(\Sigma_d, W^{\otimes d})$, one needs to determine all monomials inside $\oplus_{d \geq 0}\HH_\bu(\Sigma_d, W^{\otimes d})$ which lie in degree $i$ and have weight $d$. To each such monomial, one obtains a corresponding $GL_n(k)$-module of the form described in Corollary \ref{corgivingstructure}. In several of the applications below it is enough to know that each summand in \eqref{eq: descriptionofsummandsinchar2} and \eqref{eq: descriptionofsummandsincharodd} is of the form $S^a(W) \otimes M^{(1)}$.
\end{rem}

\begin{rem} \label{rem:tensortwists} We remark that as a $GL_{n}(k)$-module, $[V^{(b)}]^{\otimes d} \cong [V^{\otimes d}]^{(b)}$
because if $F^{b}:GL_{n}(k)\rightarrow GL_{n}(k)$ is the $b$th iteration of the Frobenius map then the action of
$G$ on $[V^{(b)}]^{\otimes d}$ is given by
\begin{equation*}
g.(v_{1}\otimes \dots \otimes v_{d})=(g.v_{1}\otimes \dots \otimes g.v_{d})
=(F^{b}(g)v_{1} \otimes \dots \otimes F^{b}(g)v_{d})=F^{b}(g)(v_{1}\otimes \dots \otimes v_{d}),
\end{equation*}
which coincides with the action of $GL_{n}(k)$ on $[V^{\otimes d}]^{(b)}$. This argument also shows that one has
natural $GL_{n}(k)$-module isomorphisms between $S^{d}(V^{(b)})\cong S^{d}(V)^{(b)}$ and
$\Lambda^{d}(V^{(b)})\cong \Lambda^{d}(V)^{(b)}$.
\end{rem}

\section{Cohomology between Young modules}

\label{sec: cohomology between Young modules}
\subsection{}
The remainder of the paper consists of applications of Theorem
\ref{thm:shiftedfinalanswer for p=2 and p oddL} and Corollary
\ref{corgivingstructure} to the representation theory of the
symmetric group. Since any $n \geq d$ gives the same results, we
will assume henceforth that $n=d$. Further, we henceforth always let $V$ denote the natural $GL_d(k)$ module, as we will be applying Theorem \ref{thm:shiftedfinalanswer for p=2 and p oddL} and Corollary \ref{corgivingstructure} but will have no need of the more general results where $V$ is not concentrated in degree zero. Theorem \ref{thm:multl of simples
in Ljgivescohomology}(b) will allow us to compute $\Hs Y^\lambda)$
from our description of $\HH_\bu(\Sigma_d, \vt)$. However much more
information can be obtained from this calculation.

In this section, we observe that with additional knowledge about
decomposition numbers for the Schur algebra $S(d,d)$  one can
compute $\Ext^\bu_{\Sigma_d}(Y^\lambda, Y^\mu)$ for arbitrary
partitions $\lambda, \mu \vdash d$. The results in this section do
not use the topological information from the previous sections.

\subsection{}
\label{subsect: tensorprodcutyoungmodules} The decomposition numbers
for $S(d,d)$ are precisely the multiplicities $[V(\lambda):L(\mu)]$
for $\lambda,\mu\vdash d$ where $V(\lambda)$ is the Weyl module of
highest weight $\lambda$. We will first explain the connection
between the decomposition numbers for $S(d,d)$ and homomorphisms
between Young modules.

\begin{prop}
\label{prop:dec/homspaces} Knowing the decomposition matrix for the
Schur algebra $S(d,d)$ is equivalent to knowing $\dim \operatorname
{Hom}_{\Sigma_{d}}(Y^{\lambda},Y^{\mu})$ for all $\lambda, \mu
\vdash d$.
\end{prop}

\begin{proof} Recall from \cite[7.1]{DPS} that ${\mathcal G}_{\text{Hom}}(Y^{\mu})=P(\mu)$, where $P(\mu)$ is the indecomposable projective cover of $L(\lambda)$ in the category of $S(d,d)$-modules. This implies that
$$\operatorname{Hom}_{\Sigma_{d}}(Y^{\lambda},Y^{\mu})\cong
\operatorname{Hom}_{S(d,d)}(P(\lambda),P(\mu)).$$ for all $\lambda,
\mu \vdash d$. The dimensions of the latter $\text{Hom}$-spaces give
the Cartan matrix for $S(d,d)$ (i.e., multiplicities of composition
factors in projective indecomposable modules). It is well known that
the decomposition matrix times its transpose gives the Cartan matrix
for $S(d,d)$. Therefore, from the decomposition matrix we can get
the dimensions of the homomorphisms between Young modules.

It remains to prove that from the Cartan matrix for $S(d,d)$ one can
deduce the decomposition numbers. This will follow by using
reverse induction on the ordering on the weights, which in this setting can
be taken as just the lexicographic order on partitions of $d$. Note
for a maximal weight $\lambda$, $P(\lambda)=V(\lambda)$ and
$$[V(\lambda):L(\mu)]=[P(\lambda):L(\mu)]$$
is known for all $\mu\vdash d$. In general, $P(\lambda)$ has a
filtration by Weyl modules with one copy of $V(\lambda)$ on top, and
successive subquotients of the form $V(\mu)$ for $\mu > \lambda$. We
let $[P(\lambda) : V(\mu)]$ denote the multiplicity of $V(\mu)$ in
any such filtration, which is well-defined by \cite[4.19]{Jantzen}.

Now assume that $[V(\sigma):L(\mu)]$ is known for all
$\sigma>\lambda$, $\mu \vdash d$. We want to be able to deduce
$[V(\lambda):L(\mu)]$ for $\mu \vdash d$. Observe that
\begin{eqnarray*}
[P(\lambda):L(\mu)]&=&\sum_{\sigma\vdash d}[P(\lambda):V(\sigma)][V(\sigma):L(\mu)]\\
&=&[P(\lambda):V(\lambda)][V(\lambda):L(\mu)]+
\sum_{\sigma>\lambda}[P(\lambda):V(\sigma)][V(\sigma):L(\mu)]\\
&=&[V(\lambda):L(\mu)]+\sum_{\sigma>\lambda}[V(\sigma):L(\lambda)][V(\sigma):L(\mu)].
\end{eqnarray*}
The last equality follows by using the reciprocity law
\cite[4.5]{Martinbook}:
$[P(\lambda):V(\sigma)]=[V(\sigma):L(\lambda)]$, and
$[P(\lambda):V(\lambda)]=[V(\lambda):L(\lambda)]=1$. It follows that
from the equation above that $[V(\lambda):L(\mu)]$ can be computed
using the induction hypothesis and the Cartan matrix.
\end{proof}

\subsection{} According to Proposition~\ref{prop:dec/homspaces}, knowing the dimensions of
Hom-spaces between Young modules for $\Sigma_d$ is equivalent to
knowing the decomposition numbers for $S(d,d)$. We will now
demonstrate a striking  result: it is enough to know
$\text{Hom}_{\Sigma_{t}}(Y^{\rho},Y^{\tau})$ for all $\rho, \tau
\vdash t \leq d$ (or equivalently the decomposition numbers for all
$S(t,t)$, $t \leq d$) in order to compute
$\text{Ext}^{\bullet}_{\Sigma_{d}}(Y^{\lambda},Y^{\mu})$ for all
$\lambda, \mu \vdash d$. This relies heavily on our computation of
$\HH_n(\Sigma_{d}, V^{\otimes d})$ as a $GL_{d}$-module.

It is an easy consequence of Mackey's theorem that the tensor
product of two Young modules is a direct sum of Young modules. The direct sum decomposition can be determined from the decomposition matrix of $S(d,d)$:

\begin{prop}
\label{prop: computingtensorofYoungmodules} Suppose the
decomposition matrix for the Schur algebra $S(d,d)$ is known. Then
for $\lambda, \mu \vdash d$, one can compute the decomposition of
$Y^\lambda \otimes Y^\mu$ into Young modules.
\end{prop}

\begin{proof} The Young modules have a filtration by Specht modules with multiplicities determined by decomposition
numbers for $S(d,d)$. In particular $Y^\lambda$ has a Specht
filtration with submodule $S^\lambda$ and other successive
subquotients of the form $S^\mu$ with $\mu>\lambda$, where
$[Y^\lambda : S^\mu]=[V(\mu) : L(\lambda)]$, see
\cite[4.6.4]{Martinbook}. Thus the decomposition numbers of $S(d,d)$
together with the ordinary character table of the symmetric group
$\Sigma_d$ (which is easily computed), allow one to compute the
ordinary character of $Y^\lambda$:
\begin{equation}
\label{eq: ordinarycharYoung} \operatorname{ch} Y^\lambda =
\sum_{\mu \geq \lambda} [V(\mu) : L(\lambda)]\chi^\mu
\end{equation}
where $\chi^\mu$ denotes the ordinary (irreducible) character of $S^\mu$.
Multiplying the two characters together gives us:
\begin{equation}
\label{eq: charoftensorprod} \operatorname{ch}(Y^\lambda \otimes
Y^\mu) = \sum m_\tau \chi^\tau.
\end{equation}
Finally, the triangular nature of the decomposition matrix of
$S(d,d)$ allows one to recover, from the character of $Y^\lambda
\otimes Y^\mu$, the direct sum decomposition of $Y^\lambda \otimes
Y^\mu$ into Young modules, using \eqref{eq: ordinarycharYoung} and
\eqref{eq: charoftensorprod}.
\end{proof}

Define numbers $g_{\sigma}^{\lambda, \mu}$ by

\begin{equation}
\label{eq: defofgsigmalambdamu} Y^{\lambda}\otimes Y^{\mu}=
\bigoplus_{\sigma \vdash d} g_{\sigma}^{\lambda, \mu} \,\,Y^\sigma.
\end{equation}
Then we have the following:

\begin{eqnarray}
\nonumber
\dim \Ext^{n}_{\Sigma_{d}}(Y^{\lambda},Y^{\mu})&=&\dim \Ext^{n}_{\Sigma_d}(k, Y^\lambda \otimes Y^\mu) \text{ since } Y^\lambda \text{ is self-dual,}\\
\label{eq: formulaforextylambdaymjintermsofkylambda}
&=&\sum_{\sigma\vdash d} g_{\sigma}^{\lambda,\mu}\dim
\Ext^{n}_{\Sigma_{d}}(k,Y^\sigma)\\&=& \sum_{\sigma\vdash d}
g_{\sigma}^{\lambda,\mu}[\HH_n(\Sigma_{d}, V^{\otimes d}):L(\sigma)]
\text{ by Theorem } \ref{thm:multl of simples in
Ljgivescohomology}(b).\nonumber
\end{eqnarray}

We can collect the previous results into a  theorem stating that
knowing only the dimension of spaces of homomorphisms between Young
modules, one can compute  $\Ext^i_{\Sigma_d}(Y^\lambda, Y^\mu)$ for
all $i \geq 0$.  We wish to use Theorem  \ref{thm:multl of simples
in Ljgivescohomology}(b) and Corollary \ref{corgivingstructure} to
compute $\dim \Ext^i_{\Sigma_d}(Y^\lambda, Y^\mu)$. The composition
factors of $S^a(V^{(b)})$ are known from work of Doty
\cite{Dotysympower}. The module $\Lambda^a(V)$ is irreducible,
isomorphic to $L(1^a)=L((1,1, \ldots, 1))$. Thus $\Lambda^a(V^{(b)})
\cong L((p^b,p^b, \ldots, p^b))$. Thus the only remaining difficulty
is computing the composition multiplicities in the tensor products
of irreducibles. Since these can be determined from decomposition
numbers, we obtain the following:

\begin{thm}
\label{thm: Homdeterminesallcohomology} Suppose $\di
\Hom_{\Sigma_t}(Y^\rho, Y^\tau)$ is known for all $t \leq d$ and for
all $\rho, \tau \vdash t$. Then there is an algorithm to compute
$\dim \Ext^i_{\Sigma_d}(Y^\lambda, Y^\mu)$ for all $i \geq 0$ and
all $\lambda, \mu \vdash d$.
\end{thm}
\begin{proof}
Suppose all the dimensions of the $\Hom$ spaces are known. Then by
Proposition \ref{prop:dec/homspaces}, all the decomposition matrices
for the Schur algebras $S(t,t)$ can be computed for $t \leq d$. The
Littlewood-Richardson rule lets one compute the multiplicities in a
Weyl filtration of $V(\lambda) \otimes V(\mu)$. The unitriangular
nature of the decomposition matrix of $S(d,d)$ then lets one compute
the composition factor multiplicities in tensor products of
irreducible modules.  This is the step that requires the
decomposition numbers not just for $S(d,d)$, but for all $S(t,t)$
with $t \leq d$, since we will be computing tensor products of
$L(\mu)$'s with $\mu \vdash t \leq d$. These multiplicities are all
that is needed to determine the multiplicities in
$\HH_{\bu}(\Sigma_d, \vt)$, since the multiplicities in
$S^a(V^{(b)})$ and $\Lambda^a(V^{(b)})$ are known, so we can use
Corollary \ref{corgivingstructure} to compute $\Ext^i_{\sd}(k,
Y^\lambda)$. Now Proposition \ref{prop:
computingtensorofYoungmodules} says one can obtain all the
$\Ext^i_{\sd}(Y^\lambda, Y^\mu)$.
\end{proof}
In the next section we will illustrate this algorithm in an example,
see Table \ref{table: extbetweenyoungmodulesd=6} for the result.

\section{A complete example: $d=6,\,\, p=2$}
\label{sec: completeexampled=6}
\subsection{}

Before proving some general results, we apply the description of
$\HH_\bu(\Sigma_d, \vt)$  and Theorem \ref{thm:multl of simples in
Ljgivescohomology} to determine $\HH^\bu(\Sigma_d, Y^\lambda)$ for
all $\lambda \vdash 6$ in characteristic two. In the next section we describe a method
for determining these groups for $\lambda \vdash d$ where $d$ is
arbitrary and $p=2$.   In general the method requires one to compute
some tensor products of simple modules for $GL_t(k)$ where $t$ is
considerably smaller than $d$. For example to compute Young module
cohomology for  $d=16$ and $p=2$ in all degrees the most difficult
computation is that of $L(2,2,1) \otimes L(1^{3})$ for $GL_8(k)$,
which is easily handled.

We then demonstrate how, using the known decomposition matrices for
$S(6,6)$ and the character table of $\Sigma_6$, one can use the
results of Section \ref{sec: cohomology between Young modules} to
determine $\Ext^\bu_{\Sigma_6}(Y^\lambda, Y^\mu)$ for arbitrary
$\lambda, \mu \vdash 6$.

\subsection{}The first step is to determine the structure of the $GL_6(k)$-module $\HH_\bu(\Sigma_6, V^{\otimes 6})$. This $V$ will be concentrated in degree zero, so we will be using the description given in Theorem \ref{thm:shiftedfinalanswer for p=2 and p oddL}(a) as a polynomial algebra with certain generators. We can first determine all possible ``shapes" for monomials which contribute to $d=6$ in the description of $\oplus_{d \geq 0}\HH_\bu(\Sigma_d, V^{\otimes d})$ in Theorem \ref{thm:shiftedfinalanswer for p=2 and p oddL}. Then for each monomial shape it is easy to work out the corresponding $GL_6(k)$-module structure from the formulas given in Section \ref{sec:The.prime.two} and the fact that we have a polynomial algebra. For example, monomials of the
form $\overline{Q}_2(v)\overline{Q}_3(v)\cdot v \cdot v$ give a
summand of $\HH_5(\Sigma_6, \vt)$ which is isomorphic to $V^{(1)}
\otimes V^{(1)} \otimes S^2(V)$ as a $GL_6(k)$-module. Monomials of
the form $\overline{Q}_{1,3}(v) \cdot v \cdot v$ give a
$GL_6(k)$-module isomorphic to $V^{(2)} \otimes S^2(V)$ as a
direct summand of $\HH_7(\Sigma_6,\vt)$. For $d=6$ and $p=2$ the possible monomials are listed in Table \ref{table:summandsind=6}, along with the corresponding $GL_6(k)$-module structure. In all such formulas, one should assume the $``v"$'s are arbitrary vectors in $V$, not assumed to be equal.\\

\begin{table}[here]
\begin{tabular}{|llll|}
  \hline
  Monomial shape & &$GL_6(k)$ structure & Degree $\bu$ \\
  \hline
  $\overline{Q}_i(v)\cdot v^4$ & $1 \leq i$& $V^{(1)} \otimes S^4(V)$ & $i$ \\
  $\overline{Q}_{i,j}(v) \cdot v^2$ &$1 \leq i \leq j$  & $V^{(2)} \otimes S^2(V)$ & $i+2j$ \\
  $\overline{Q}_{i}(v)\cdot \overline{Q}_j(v) \cdot v^2$ & $1\leq i<j$ &$V^{(1)} \otimes V^{(1)} \otimes S^2(V)$&$i+j$\\
  $\overline{Q}_i(v) \cdot \overline{Q}_i(v) \cdot v^2$ & $1 \leq i $& $S^2(V^{(1)}) \otimes S^2(V)$ & $2i$ \\
  $\overline{Q}_{i,j}(v) \cdot \overline{Q}_k(v)$ & $1 \leq i \leq j, \,\,1 \leq k$
   & $V^{(2)} \otimes V^{(1)}$ & $i+2j+k$ \\
  $\overline{Q}_i(v) \cdot \overline{Q}_j(v) \cdot \overline{Q}_k(v)$ & $1 \leq i < j < k$ &
  $V^{(1)}\otimes V^{(1)} \otimes V^{(1)}$ & $i+j+k$ \\
  $\overline{Q}_i(v) \cdot \overline{Q}_i(v) \cdot \overline{Q}_j(v)$ &$1 \leq i, j  \,\,\, i \neq j$& $S^2(V^{(1)})\otimes V^{(1)}$ & $2i+j$ \\
  $\overline{Q}_i(v) \cdot \overline{Q}_i(v) \cdot \overline{Q}_i(v)$ &$1 \leq i$& $S^3(V^{(1)})$ & $3i$ \\

  \hline
\end{tabular}
\vspace*{0.1in} \caption{$GL_6(k)$-module summands of
$\HH_\bu(\Sigma_6, V^{\otimes 6})$, $\bu>0$.}
\label{table:summandsind=6}
\end{table}

The modules which occur in Table \ref{table:summandsind=6} (and
indeed for any $d$ in characteristic two) are all tensor products of
Frobenius twists of symmetric powers of the natural module. In
\cite{Dotysympower}, Doty determined the entire submodule structure
of $S^m(V)$. It is always multiplicity free, and there is a simple
rule for determining which $L(\lambda)$ appear. In the Grothendieck
group we have:

\begin{equation}
    \begin{array}{l}
  [S^1(V)]=[L(1)]\\

 [S^2(V)]=[L(2)]+[L(1^2)] \\

  [S^3(V)]=[L(3)]+[L(1^3)] \\

 [S^4(V)]=[L(4)]+[L(3,1)]+[L(2^2)]+[L(1^4)]. \\
 \end{array}
\end{equation}

To calculate the tensor products in Table
~\ref{table:summandsind=6}, one can often make use of the Steinberg
Tensor Product Theorem (STPT), but we need some information on tensor products of simple modules. The next lemma follows from \cite{KleshSheth}, where the entire submodule structure of the Weyl module $V(2^a,1^b)$ is determined in any characteristic.

\begin{lem}
\label{lem: twopartWeylmodulesstructure}Let $p=2$ and let $a \leq n$. Then, in the Grothendieck group we have:
\begin{equation}
V(2,1^{a-2})=
\begin{cases} [L(2,1^{a-2})] & \text{if $a$ is odd,}
\\
[L(2,1^{a-2})] + [L(1^a)] &\text{if $a$ is even.}
\end{cases}
\end{equation}
\begin{equation}
V(2,2,1^{a-4})=
\begin{cases} [L(2,2,1^{a-4})] + [L(2,1^{a-2})]& \text{if $a \equiv 0$ mod 4,}
\\
 [L(2,2,1^{a-4})] + [L(1^{a})] &\text{if $a \equiv 1$ mod 4,}
 \\
 [L(2,2,1^{a-4})] + [L(2,1^{a-2})]+[L(1^{a})] &\text{if $a \equiv 2$ mod 4,}\\
  [L(2,2,1^{a-4})] &\text{if $a \equiv 3$ mod 4.}
\end{cases}
\end{equation}
\end{lem}
The Littlewood-Richardson rule can be used to compute the multiplicities in a Weyl filtration of $V(1^a) \otimes V(1^b) \cong L(1^a) \otimes L(1^b)$, and all terms have the form $V(2^x, 1^y)$. Since the composition factor multiplicities in these ``two-column" Weyl modules are all known, one can determine the composition factor multiplicities in any $L(1^a) \otimes L(1^b)$. We will only need those in the following lemma, which follows easily from the Littlewood-Richardson rule and Lemma \ref{lem: twopartWeylmodulesstructure}.

\begin{lem}
\label{lem: omnibus1a1btensorlemma}
Suppose $p=2$ and let $n \geq a+1$ in the first case and $n \geq a+2$ in the second. Then in the Grothendieck group we have:
\label{lem: tensorproducts11x11111}

\begin{equation*}
[L(1) \otimes L(1^a)]=
\begin{cases}
[L(2,1^{a-1})] + [L(1^{a+1})] & \text{ if $a \equiv 0$ mod 2}\\
[L(2,1^{a-1})]+ 2[L(1^{a+1})] & \text{ if $a \equiv 1$ mod 2}
\end{cases}
\end{equation*}

\begin{equation*}
 [L(1^2) \otimes L(1^a)]=\begin{cases}
  [L(2^2,1^{a-2})]+ 2[L(2,1^a)]+ 3[L(1^{a+2})] & \text{ if $a \equiv 0$ mod 4}\\
  [L(2^2,1^{a-2})]+ [L(2,1^a)] + [L(1^{a+2})] & \text{ if $a \equiv 1$ mod 4}\\
  [L(2^2,1^{a-2})]+ 2 [L(2,1^a)] + 2[L(1^{a+2})] & \text{ if $a \equiv 2$ mod 4}\\
   [L(2^2,1^{a-2})]+ [L(2,1^a)] + 2 [L(1^{a+2})] & \text{ if $a \equiv 3$ mod 4}
   \end{cases}
   \end{equation*}
\end{lem}

 For this section we only use the following special case of Lemma \ref{lem: omnibus1a1btensorlemma}:

\begin{equation}
    \label{eq: easy tensor products}
        \begin{array}{l}
           [ L(1) \otimes L(1)]=[L(2)] + 2[L(1^2)]\\

   [L(1^2) \otimes L(1)]=   [L(2,1)]+[L(1^3)]. \\
        \end{array}
\end{equation}

\begin{rem}
\label{rem: remarkabouttensorproductsofsimples} Since we are assuming $n=d$, the tensor products $L(\lambda) \otimes L(\mu)$ which arise in \eqref{eq: descriptionofsummandsinchar2} and \eqref{eq: descriptionofsummandsincharodd} only occur when $\lambda \vdash a, \mu \vdash b$ with $a+b \leq d$. Thus \emph{all statements about such tensor products of polynomial $GL_d(k)$-modules will implicitly assume that $d$ is large}. For example \eqref{eq: easy tensor products} would be false if $n=2$, in this case there is no module $L(1^3)$. However for us writing $L(1^2) \otimes L(1)$ implies $n$ is at least three. Similarly Lemma \ref{lem: omnibus1a1btensorlemma} would be false without the assumptions on $n$ and $a$.
\end{rem}

Using \eqref{eq: easy tensor products} and the STPT, we can
calculate the tensor products in Table \ref{table:summandsind=6} and
determine the composition multiplicities in each summand. These are
given in Table \ref{table: constituentsofsummandd=6case}.

\begin{table}[here]
\begin{tabular}{|l|cccccc|}
\hline
     &$L(6)$&$L(51)$&$L(42)$&$L(33)$&$L(31^3)$&$L(2^3)$ \\
\hline  $V^{(1)} \otimes S^4(V)$&1&1&1&2&1&1\\
  $V^{(2)} \otimes S^2(V)$&1&1&0&0&0&0\\
  $V^{(1)}  \otimes V^{(1)}  \otimes S^2(V)$&1&1&2&2&0&2\\
  $S^2(V^{(1)}) \otimes S^2(V)$&1&1&1&1&0&1\\
  $V^{(2)} \otimes V^{(1)}$&1&0&0&0&0&0\\
  $V^{(1)}  \otimes V^{(1)}  \otimes V^{(1)}$&1&0&2&0&0&2\\
  $S^2(V^{(1)} ) \otimes V^{(1)} $&1&0&1&0&0&1\\
  $S^3(V^{(1)})$&1&0&0&0&0&1\\
  \hline
\end{tabular}
\vspace*{0.1in} \caption{Simple module multiplicities in direct
summands of $\HH_\bu(\Sigma_6, V^{\otimes 6}).$} \label{table:
constituentsofsummandd=6case}
\end{table}
\subsection{}The final step is to determine how many copies of each direct summand appear in each degree. Here is an example:

\begin{lem}
\label{lemma:multiplicityofsamplesummandsd=6} The $GL_6(k)$ module
$V^{(2)} \otimes S^2(V)$ appears
$\lfloor{\frac{t-1}{2}}\rfloor-\lfloor{\frac{t-1}{3}}\rfloor$ times
in $\HH_t(\Sigma_6,V^{\otimes 6})$, where $\lfloor \rfloor$ is the
greatest integer function.
\end{lem}

\begin{proof} From line two in Table \ref{table: constituentsofsummandd=6case}, the multiplicity in degree $t$ will be the number of ways to write $t=i+2j$ with $1 \leq i \leq j$, equivalently to write $t-1=i+2j$ with $0 \leq i<j$. Expressing $j$ as $i+c$ for $c > 0$, we immediately see this is the number of ways to write $t-1$ as the sum of a positive even integer and a nonnegative multiple of $3$. This is just sequence number A008615 in \cite{Intseq}, and the formula given is one of several possible. For example when $t=9$ the monomials of shapes $\overline{Q}_{1,4}(v) \cdot v^2$ and
$\overline{Q}_{3,3}(v) \cdot v^2$ contribute the two such summands.
\end{proof}

The other summands can be handled similarly and, since $d$ is fairly
small,  have reasonable closed form multiplicity formulas. For
instance $V^{(2)}\otimes V^{(1)} $ will occur in degree $t$ with
multiplicity the number of ways to write $t$ as $i+2j+k$ with $i
\leq j$. One fairly easily determines this is $\lfloor
\frac{(t-1)^2}{12}\rfloor$, the number of partitions of $t-1$ with
exactly 3 parts.

We now have all the information to determine the nontrivial Young
module cohomology, the remaining calculations we leave to the
reader. For notational convenience in the formulas below set \
\begin{equation}
\label{eq: defofa} a=a(j)=\frac{(j+1)(j+2)}{6}
\end{equation}
\ and let $\lceil \rceil$ be the ceiling function.

\begin{thm}
\label{th: Young module cohomologyford=6} Let $p=2$. The dimensions
of the nonzero cohomology groups $\HH^j(\Sigma_6, Y^\lambda)$ are
given by the following table:

\begin{table}[here]
    \begin{tabular}{|l|l|l|l|l|}
\hline    $\lambda$ & $\dim \HH^j(\Sigma_6,
Y^\lambda)$&\,\,\,\,\,&$\lambda$ & $\dim \HH^j(\Sigma_6,
Y^\lambda)$\\ [1ex] \hline
  $(6)$ & $\lceil a \rceil $&\,\,\,\,\,&  $(5,1)$ & $1+\lfloor \frac{2j}{3} \rfloor$\\[1ex]
     $(4, 2)$ & $\lfloor a \rfloor $&\,\,\,\,\,& $(3,3)$ & $j+1$\\[1ex]
      $(3,1^3)$ & $1$&\,\,\,\,\,&   $(2^3)$& $\lceil a \rceil$ \\\hline
    \end{tabular}
    \vspace*{0.1in}
    \caption{Cohomology of Young modules for $\Sigma_6$, p=2}
    \label{table:finalanswerd=6}
    \end{table}
\end{thm}

\begin{cor}
\label{cor: 411innucleus} In characteristic two, $\HH^j(\Sigma_6, Y^{(4,1^2)})=0$ for all
$j\geq 0$.
\end{cor}

Observe that $Y^{(4,1^2)}$ is the only nonprojective Young module in
the principal block with this property. Thus its support variety
contributes to the so called {\it representation theoretic nucleus}.
This corollary could be obtained ``by hand" but it is already quite
complicated. The module $Y^{(4,1^2)}$ is 48-dimensional. It is
uniserial of length 9 and its projective cover is a nonsplit
extension of $Y^{(4,1^2)}$ by itself. Thus its projective resolution
is periodic. In Theorem \ref{thm: Young modules with no cohomology}
we will give a necessary and sufficient condition on $\lambda$ for vanishing of
$\HH^\bullet(\Sigma_d, Y^\lambda)$.

We also know that $Y^{(6)} \cong k$ and remark the formula for
$\HH^\bu(\Sigma_6, Y^{(6)})\cong \HH^\bu(\Sigma_6, k)$ agrees with
previous results, for instance \cite[IV.5]{AMcohomologyfinitegroups}
where the entire ring structure of $\HH^\bu(\Sigma_6, k)$ is
determined.

\subsection{}For $d=6$ and $p=2$ we implemented the procedure described in Section \ref{sec: cohomology between Young modules}. The tensor product decompositions together with Theorem \ref{th: Young module cohomologyford=6} allow one to determine all the $ \Ext^i_{\Sigma_6}(Y^\lambda,Y^\mu)$.  (Notice the results of Theorem \ref{th: Young module cohomologyford=6} are needed in the computation \eqref{eq: formulaforextylambdaymjintermsofkylambda}.) All nonvanishing cohomology between all Young modules for $\Sigma_6$ is given in Table \ref{table: extbetweenyoungmodulesd=6}.

\begin{table}[here]
\begin{tabular}{|l|lllllll|}
\hline
    $\lambda / \mu$&$6$&$51$&$42$&$41^2$&$3^2$&$31^3$&$2^3$ \\
 \hline
   $6$&$\lceil a \rceil$&&&&&&\\

   $51$& $1 +\lfloor \frac{2j}{3} \rfloor$&$2 +2\lfloor \frac{2j}{3} \rfloor$&&&&&\\

$42$&$\lfloor a \rfloor$&$j+1$& $1+ \lfloor a \rfloor +\lceil a \rceil $&&&&\\

$41^2$&0&1&2&3&&&\\

$3^2$&$j+1$&$2j+2$&$2j+2$&2&$4j+4$&&\\

$31^3$&1&2&2&2&4&4&\\

$2^3$&$\lceil a \rceil$&$j+1$&$1+2\lfloor a
\rfloor$&2&$2+2j$&2&$1+2\lceil a \rceil$\\ \hline
\end{tabular}
\vspace*{0.1in} \caption{Dimensions of $\Ext^j_{\Sigma_6}(Y^\lambda,
Y^\mu)$, $p=2$} \label{table: extbetweenyoungmodulesd=6}
\end{table}
Notice that $\lfloor a \rfloor$ and  $\lceil a \rceil$ differ only
when $j$ is a multiple of 3, reflecting the contribution of the
$S^3(V^{(1)})$ terms in these degrees.

\subsection{}Recall that the {\it complexity} of a module $M$ can be defined as the minimal $c$ such that there exists a constant $K>0$ with:

$$\dim \Ext^j_{\sd}(M,M) \leq Kj^{c-1}.$$
Notice from \eqref{eq: defofa} that $a$ is quadratic in $j$. Thus
the diagonal entries in Table \ref{table: extbetweenyoungmodulesd=6}
prove that $Y^{(6)}$, $Y^{(4,2)}$ and $Y^{(2,2,2)}$ have complexity
3, while $Y^{(5,1)}$ and $Y^{(3.3)}$ have complexity 2. Finally,
$Y^{(4,1,1)}$ and $Y^{(3,1,1,1)}$ have complexity 1 while the
remaining Young modules are all projective. These results agree with
\cite{HNsuppvar}, where the complexity of any Young module in any
characteristic was determined.

\subsection{} Although the computations are too long to include in detail, we will attempt to convince the reader that even for $d \leq 16$, one can apply this method to determine  $\HH^i(\Sigma_d, Y^\lambda)$ for any $i$ and any $\lambda \vdash d$. For $d=16$ one can quickly write down the equivalent of Table \ref{table:summandsind=6}, it is just much larger. Each corresponding
$GL_{16}(k)$-module  will be a tensor product of modules of the form $S^a(V^{(b)})$. Since the constituents of these modules are all known, the only obstacle is calculating the tensor products of the irreducibles which occur. The decomposition matrices for the Schur algebra $S(n,d)$ in characteristic 2 are well-known for $d$ up to at least $10$, for example, see the
appendix of \cite{Mathasbook}.  Most of the cases reduce using the STPT to very small computations.  For the few difficult cases one can use the Littlewood-Richardson rule to compute the tensor products of Weyl modules, and then use the decomposition matrices to handle the simple modules .  One of the ``larger" cases which occurs is calculating the composition factor multiplicities of
$L(2^2,1) \otimes L(1^3)$ as a $GL_8(k)$-module. This computation
arises from the $L(4^2,2) \otimes L(2^3)$ inside the summand
$$V^{(1)} \otimes V^{(1)} \otimes V^{(1)}\otimes V^{(1)} \otimes
V^{(1)} \otimes S^6(V)$$ corresponding to monomials of the form
$\overline{Q}_i(v)\cdot \overline{Q}_j(v)\cdot \overline{Q}_k(v)
\cdot \overline{Q}_l(v) \cdot \overline{Q}_m(v) \cdot v^6$.
 This computation is easily handled using the known decomposition matrices for the Schur algebras $S(5,5)$ and $S(8,8)$ in characteristic two, together with the Littlewood-Richardson rule.

 The power of the method is that a small number of tensor product computations for much smaller values of $d$ allows one to compute $\HH^i(\Sigma_d, Y^\lambda)$ in arbitrary degree. Of course results corresponding to Lemma \ref{lemma:multiplicityofsamplesummandsd=6}
and closed form formulas like those in Table \ref{table:finalanswerd=6} will not be
obtained. One would need to solve combinatorial  problems like ``Find a closed form formula for the number of ways to express an integer $m$ in the form $3i+2j + r+ s+ t$ where $i, j, r, s, t$ are distinct."

\section{Cohomology vanishing theorems.}
\label{sec: Cohomology vanishing theorems}
\subsection{}
In this section we prove some general results about Young module
cohomology. Sullivan \cite{sullivan} calculated the composition
factors of $S^d(V)$, in particular determining that it is
multiplicity free. Doty \cite{Dotysympower} then calculated the
entire submodule structure of $S^d(V)$ and gave a nice way to
describe the composition factors. Corollary \ref{corgivingstructure}
allows us to determine precisely which Young modules have no
cohomology, in arbitrary characteristic; the answer being closely
related to the results of Doty and Sullivan. Recall that
$\lambda=(\lambda_1, \lambda_2, \ldots, \lambda_s)$ is
$p$-restricted if $\lambda_i - \lambda_{i+1} < p$ for all $i$.
Any $\lambda \vdash d$ can be written uniquely as $\lambda_{(0)} +
p\mu$ where $\lambda_{(0)}$ is $p$-restricted and $\mu$ is a
partition. The modules $\{Y^\lambda \mid \lambda \text{ is $p$-restricted}\}$ form
a complete set of indecomposable projective (hence injective)
$k\Sigma_d$-modules.

\begin{thm}
\label{thm: Young modules with no cohomology} Let $\lambda \vdash
d$. Write  $\lambda=\lambda_{(0)} + p\mu$ with $\lambda_{(0)}\vdash
a$  p-restricted and $p=\operatorname{char} k$ arbitrary.  Then
$$\Hs Y^\lambda)\neq 0$$ if and only if $[S^a(V): L(\lambda_{(0)})] \neq 0$ or $a=0$.
\end{thm}
\begin{proof}
If $\mu= \emptyset$ then $Y^\lambda$ is injective so
$\HH^i(\Sigma_d, Y^\lambda)=0$ for $i>0$. The $i=0$ case is Proposition \ref{prop: dimensionofHomkYlambda}.

So now assume $\mu \neq \emptyset$ and suppose $Y^\lambda$ has
nonvanishing cohomology in some positive degree. By Theorem
\ref{thm:multl of simples in Ljgivescohomology},
 $L(\lambda)$ is a constituent of some summand of $\HH_\bu(\Sigma_d,\vt)$. According to  Corollary \ref{corgivingstructure}, these summands are twists of tensor products of symmetric and exterior powers of the natural module, where at most one summand, a symmetric tensor, is not twisted. Thus $L(\lambda)$ is a constituent of  $S^u(V) \otimes M^{(1)}$ for some $GL_d(k)$-module $M$. The result about $\lambda$ follows from the STPT.

Now suppose $\lambda=\lambda_{(0)} + p\mu$ as above, with either
$a=0$ or  $[S^a(V): L(\lambda_{(0)})] \neq 0$ and $\mu \vdash t$.
Since $L(\mu)$ is a constituent of $V^{\otimes t}$ then $L(\lambda)$
is a constituent of

 $$S^a(V) \otimes V^{(1)} \otimes \cdots \otimes V^{(1)} \cong  S^a(V) \otimes (V^{\otimes t})^{(1)}.$$ This module corresponds to monomials  of the form

 $$\overline{Q}_{i_1}(v) \cdot \overline{Q}_{i_2}(v) \cdots \overline{Q}_{i_t}(v) \cdot v^a  \,\,\,1 \leq i_1 <i_2 < \cdot < i_t.$$ contributing to degree $(p-1)\sum i_u$.  Choosing $i_s=s$ we see that

 $$\HH^{(p-1)t(t+1)/2}(\Sigma_d, Y^\lambda) \neq 0.$$
\end{proof}
Theorem \ref{thm: Young modules with no cohomology} guarantees that
many nonprojective Young modules have vanishing cohomology. For
example $\Sigma_{16}$ in characteristic two has 118 nonprojective
Young modules in the principal block. Theorem \ref{thm: Young
modules with no cohomology} proves that exactly 47 of them have
vanishing cohomology!

\subsection{}
For a finite group $G$ the {\em  $p$-rank} is the maximal rank of an
elementary abelian $p$ subgroup. It is also the maximal complexity
of a $G$-module in characteristic $p$ and is the complexity of the
trivial module. Let $b= \lfloor \frac{d}{p} \rfloor$, which is the
$p$-rank of $\Sigma_d$.  It follows from work of Benson
\cite{bensonnyj} that there are no modules $M$ with complexity $b$
such that $\HH^\bu(\Sigma_d,M)=0$. Furthermore, in characteristic
two there can be none of complexity $b-1$ either. The next corollary
says that among Young modules with vanishing cohomology, all other
complexities do occur.

\begin{cor}
\label{cor: Youngmodulesofalmosteverypossiblecomplexity}
 Let $p$ be arbitrary. For every $ 1 \leq c \leq b-2$, there is a Young module $Y^\lambda$ in the principal block of $k\sd$ which has complexity $c$ and such that $\Hs Y^\lambda)=0$. There is also such a Young module of complexity $b-1$ precisely when $p>2$.
\end{cor}

 \begin{proof}
Recall from \cite{HNsuppvar} that for  $\lambda=\lambda_{(0)} +
p\mu$ where $\lambda_{(0)}$ is $p$-restricted, the complexity of
$Y^\lambda$ is $c$ where $\mu \vdash c$. We will describe how to
choose $\lambda$ to satisfy the theorem.
 Consider first the case of characteristic two. There are no partitions of $2$ or $3$ which are 2-restricted, which lie in the principal block and are not of the form $(1^a)$. Thus the smallest choice of $\lambda_{(0)}$ is $(2,1,1)$, so there are no Young modules of complexity $b-1$ or $b$ with vanishing cohomology.

 In general, for $d$ even, choose $\lambda_{(0)}=(2,1^{2(b-c)-2})$ with $\mu$ arbitrary. For $d$ odd choose $\lambda_{(0)}=(2,2,1^{2(b-c)-3})$ and $\mu$ arbitrary. In both cases $Y^\lambda$ lies in the principal block and has the desired complexity $c$ and vanishing cohomology.

 Now suppose $p$ is odd and let $d=bp+s$ with $0 \leq s<p$. Choose $\lambda_{(0)}=(s,1^{ep})$ with $e \geq 1$. Then $Y^\lambda$ is in the principal block of $\Sigma_d$. Since $[S^{ep+s}(V):L(\lambda_{(0)})]=0$, the module $Y^\lambda$ has vanishing cohomology. Choosing $e=1, 2, \ldots, b$ gives modules of complexity $b-1, b-2, \ldots, 0$ respectively which have vanishing cohomology.
 \end{proof}

\section{Cohomology in low degrees}
\label{sec: low degree cohomology}
\subsection{}

In this section we compute $\HH^i(\Sigma_d, Y^\lambda)$  for $i=1,
2$ and arbitrary $d$ and $\lambda$, in characteristic two.  For
arbitrary characteristic,   $\HH^i(\Sigma_d, Y^\lambda)=0$ for $1
\leq i < 2p-3$ by \cite[Cor. 6.3]{KNcomparingcoho}. In particular
for any $\lambda$ in odd characteristic, we have $\HH^1(\Sigma_d,
Y^\lambda)=\HH^2(\Sigma_d, Y^\lambda)=0$. One could likely work out
the first few nonvanishing degrees $i=2p-3, 2p-2, 2p-1$ in the same
way we do below for $p=2$, which we leave to the reader. The case
$i=0$ is already known:

\begin{prop}\cite[6.5]{Hemmerfixedpoint}
\label{prop: dimensionofHomkYlambda}
$$\dim \Hom_{\Sigma_d}(k,Y^\lambda)=[S^d(V): L(\lambda)].$$
\end{prop}
Doty determined these multiplicities in all characteristics. For
$p=2$ they are given in Proposition \ref{prop: Dotytheoreminchar2}.

\subsection{}We first observe that Doty's result takes
on a particularly nice form in characteristic two:

\begin{prop}Let $\lambda \vdash s $ have a 2-adic expansion $$\lambda = \sum_{i=0}^m2^i\lambda_{(i)}$$ where each $\lambda_{(i)}$ is $2$-restricted. Then $L(\lambda)$ is a constituent of $S^s(V)$ if and only if each $\lambda_{(i)}$ is of the form $(1^{a_i})$ for $a_i \geq 0$.
\label{prop: Dotytheoreminchar2}
\end{prop}
\begin{proof}Let $\lambda=(\lambda_1, \lambda_2, \cdots ,\lambda_t)$ and let  $\lambda_i=\sum_{u}c_{iu}2^u$ where the $c_{iu} \in \{0,1\}$. Doty's theorem says that  $L(\lambda)$ is a constituent of $S^s(V)$ if and only if $\lambda$ is  maximal among partitions of $s$ with its \emph{carry pattern} (see \cite{Dotysympower} for definition). This is easily seen to be equivalent to the condition that whenever $c_{iu}=1$ then $c_{i1}=c_{i2}= \cdots = c_{iu}=1$. Informally, if one does the addition $\lambda_1+ \cdots + \lambda_t=s$ in binary, all the ones in each column are as far  towards the top of the column as possible. Given such a partition, it is clear that the 2-adic expansion can be read off from this addition, and it is of the form desired. For example, consider $(15,5,5,1)$. In binary the addition $15+5+5+1=26$ takes the form:

$$\begin{array}{ccccc}
  _{1} &_{2} &_{1}&_{2}&  \\
   &1&1&1&1\\
   &0&1&0&1  \\
   &0&1&0&1\\
   +&0&0&0&1\\\hline
  1&1 & 0 & 1&0 \\
\end{array}$$and this is the largest partition of 26 with carry pattern 1,2,1,2. Thus $[S^{26}(V): L(15,5,5,1)] \neq 0$. The 2-adic expansion corresponds to the columns:
$$(15,5,5,1)=(1^4)+2\cdot (1) + 4 \cdot (1^3) + 8 \cdot (1).$$ The partition $(13,7,4,1)$ has the same carry pattern but has $2$-adic expansion
$$(13,7,4,1)=(3,3,2,1)+2 \cdot (3,1) +4\cdot (1)$$ and thus $[S^{26}(V): L(13,7,4,1)]=0$.

\end{proof}

\subsection{} From Theorem \ref{thm:shiftedfinalanswer for p=2 and p oddL} we see the only monomials contributing to degree one are of the form $Q_1(v) \cdot v^{d-2}$. Thus the following is immediate.

\begin{prop}
\label{prop: Ext1ismultinVtimesSd-2(V)} Assume $p=2$. The dimension
of $\HH^1(\Sigma_d, Y^\lambda)$ is  $[S^{d-2}(V) \otimes V^{(1)}:
L(\lambda)]$
\end{prop}
Since the composition factors of $S^{d-2}(V)$ are given by
Proposition \ref{prop: Dotytheoreminchar2}, we need to calculate the
tensor product of each with $V^{(1)} \cong L(2)$. We start with a
theorem, valid in arbitrary characteristic, that lets us ignore the
$2$-restricted part of $\lambda$ in characteristic two.

\begin{thm}
\label{thm: oktostrip1^aoff}Let $p$ be arbitrary and let $\lambda=
\lambda_{(0)} + p\mu \vdash d$ where $\lambda_{(0)} \vdash a$ is
$p$-restricted. Assume $[S^a(V) : L(\lambda_{(0)})] \neq 0$. Then
$$\HH^i(\sd, Y^\lambda) \cong \HH^i(\Sigma_{d-a},Y^{p\mu})$$ as $k$-vector spaces.
\end{thm}
\begin{proof} From Corollary \ref{corgivingstructure} we see that each $GL_d(k)$-module direct summand of $\HH_\bu(\Sigma_d, \vt)$ can be written in the form $S^a(V) \otimes M^{(1)}$ for some $M$. The result follows from the  STPT.
\end{proof}

\subsection{}  Recall that if $[S^a(V) : L(\lambda_{(0)})] = 0$ then $\HH^\bu(\Sigma_d, Y^\lambda)$ is identically 0 by Theorem \ref{thm: Young modules with no cohomology}. Thus Theorem \ref{thm: oktostrip1^aoff} reduces the problem of computing $\HH^1(\sd, Y^\lambda)$ in characteristic two to the case where $\lambda=2\mu$. Next we reduce to the case where $\lambda$ is not of the form $4\mu$. From Lemma \ref{lem: omnibus1a1btensorlemma} we have:
\begin{equation}
\label{eq: tensorproduct 1x1111}
    \begin{array}{l}
[L(2) \otimes L(2^{2a})] = [L(4,2^{2a-1})] +[L(2^{2a+1})]\\

[L(2) \otimes L (2^{2a+1})] = [L(4,2^{2a})] +2[L(2^{2a+2})].\\
        \end{array}
\end{equation}
This allows us to obtain the following stability result.
\begin{thm}
\label{thm: stability result} Suppose $\lambda \vdash d$ and $p=2$.
Then:
$$\HH^1(\Sigma_{2d}, Y^{2\lambda}) \cong \HH^1(\Sigma_{4d}, Y^{4\lambda})$$ as $k$-vector spaces.
\end{thm}
\begin{proof}

We will give a bijection between appearances of $L(4\lambda)$ in
$L(2) \otimes S^{4d-2}(V)$ and appearances of $L(2\lambda)$ in $L(2)
\otimes S^{2d-2}(V).$ Since $V^{(1)} \cong L(2)$, this together with
Proposition \ref{prop: Ext1ismultinVtimesSd-2(V)} will establish the
result.

Suppose $L(4\lambda)$ is a constituent of $L(2) \otimes L(\mu)$ for
$L(\mu)$ a constituent of $S^{4d-2}(V)$. By Proposition \ref{prop:
Dotytheoreminchar2}, $\mu=(2^{a_1}) + (4^{a_2}) + (8^{a_3}) +
\cdots.$ So by the STPT and \eqref{eq: tensorproduct 1x1111}, $L(2)
\otimes L(\mu)$ does  not have any constituents of the form $L(4
\lambda)$ unless $a_1=1$. In this case

\begin{eqnarray}
\label{eq: stabilityequation}
L(2) \otimes L(\mu) & \cong & L(2) \otimes L(2) \otimes L(4^{a_2}) \otimes \cdots\\
&=&(L(4) +2 L(2,2)) \otimes  L(4^{a_2}) \otimes \cdots \nonumber
\end{eqnarray}
 Now any $L(4\lambda)$ composition factors must come from the $L(4) \otimes  L(4^{a_2}) \otimes \cdots$ in \eqref{eq: stabilityequation}. Thus we have:
 \begin{eqnarray}\nonumber
\label{eq: first stability equationcomputation} [L(2) \otimes
L(\mu): L(4\lambda)] &=&[L(4) \otimes  L((4^{a_2}+ 8^{a_3} +
\cdots)) : L(4 \lambda)]\\ &=& [(L(2) \otimes  L((2^{a_2}+
4^{a_3} + \cdots)))^{(1)} : L(2 \lambda)^{(1)}]\\\nonumber &=&[L(2)
\otimes  L((2^{a_2}+ 4^{a_3} + \cdots)) : L(2 \lambda))].
\end{eqnarray}
Thus each $L(4\lambda)$ in $L(2) \otimes L(\mu)$ corresponds to an
$L(2\lambda)$ in $L(2) \otimes L((2^{a_2}+ 4^{a_3} + \cdots)) $
inside $L(2) \otimes S^{2d-2}(V)$ and similarly in reverse, so the
multiplicities are the same and the result follows.
\end{proof}

The previous two results are now enough to completely determine
$\HH^1(\Sigma_d, Y^\lambda)$:

\begin{thm}
\label{thm: determineext1(k,Ylambda)}

Let $p=2$ and let $\lambda \vdash d$ be in the principal block with
$\lambda$ not 2-restricted (otherwise $Y^\lambda$ is projective).
Write the 2-adic expansion of $\lambda$ as

  $$\lambda =\lambda_{(0)} + 2^s\lambda_{(s)} + 2^{s+1}\lambda_{(s+1)} + \cdots +2^r \lambda_{(r)}$$
 where $ s \geq 1$, $\lambda_{(s)} \neq \emptyset$ and each $\lambda_{(i)}$ is 2-restricted. Then:

\begin{enumerate}

\item[(a)]  If $\{\lambda_{(i)} : i \neq s\}$ are all of the form $(1^{a_i})$ and $\lambda_{(s)}=(2,1^a)$ or $(1^b)$ with $b$ odd then
  $${\rm dim}\HH^1(\Sigma_d, Y^\lambda)=1.$$
  \item[(b)]  If $\{\lambda_{(i)} : i \neq s\}$ are all of the form $(1^{a_i})$ and if $\lambda_{(s)}= (1^b)$ with $b$ even then  $${\rm dim}\HH^1(\Sigma_d, Y^\lambda)=2.$$
  \item[(c)] Otherwise $\HH^1(\Sigma_d, Y^\lambda)=0.$
\end{enumerate}
\end{thm}

\begin{proof}Since any constituent of $S^{d-2}(V)$ is of the form $L(\tau)$ for $\tau=(1^a) + 2\mu$, it is immediate that $\lambda_{(0)}$ must be of the form $(1^a)$ for the cohomology to be nonzero.
 Now Theorem \ref{thm: oktostrip1^aoff} says we can assume without loss that $\lambda_{(0)}= \emptyset.$
With this assumption we can apply Theorem \ref{thm: stability
result} to assume, again without loss of generality, that $s=1$,
i.e. $\lambda$ is of the form $2\mu$ but not of the form $4\mu$.

We know $\HH^1(\sd, Y^\mu)$ is nonzero if and only if $L(\mu)$
occurs in $L(2) \otimes S^{d-2}(V)$. Referring to \eqref{eq:
stabilityequation}, and our assumption that $\lambda$ is not of the
form $4\mu$, we need to know when $L(\lambda)$ occurs inside
$$L(2) \otimes L(2^a) \otimes L(4^{a_2}) \otimes \cdots.$$ Now the STPT and \eqref{eq: tensorproduct 1x1111} gives the result.

\end{proof}

\begin{exm}In characteristic two, $\HH^1(\Sigma_{47}, Y^{(17,13,13,4)}) \cong k$.
\end{exm}
This follows since $(17,13,13,4)=(1^3)+ 2^2(2,1,1)+ 2^3(1^3)$ so $s=2$
and we are in case (a) of the theorem. Similarly since
$(57,41,9,8)=(1^3)+2^3(1^4)+2^4(1)+2^5(1^2)$ we conclude that
\begin{exm}In characteristic two, $\dim \HH^1(\Sigma_{115}, Y^{(57,41,9,8)})=2.$
\end{exm}

\subsection{}
Next we compute $\HH^2(\Sigma_d, Y^\lambda)$ in characteristic two.

\begin{prop}
\label{prop: dimext2ismultinsummands} Let $p=2$. The dimension of
$\HH^2(\Sigma_d, Y^\lambda)$ is the multiplicity of $L(\lambda)$ in
$$L(2) \otimes S^{d-2}(V) \oplus S^2(V^{(1)}) \otimes S^{d-4}(V).$$
\end{prop}
\begin{proof}The only monomials which contribute to degree two are those of the form $Q_2(v) \cdot v^{d-2}$ and $Q_1(v)\cdot Q_1(v) \cdot v^{d-4}$.
\end{proof} The multiplicity of $L(\lambda)$ in $L(2) \otimes S^{d-2}(V)$ is $\dim \HH^1(\Sigma_d,Y^\lambda)$, which we have already determined.
Theorem \ref{thm: oktostrip1^aoff} makes it sufficient to determine
$\HH^2(\sd, Y^{2\lambda})$.

Next we prove another stability result:
\begin{thm}
\label{thm: stability resultforext2} Let $p=2$ and suppose $\lambda \vdash d$.
Then
$$\HH^2(\Sigma_{4d}, Y^{4\lambda}) \cong \HH^2(\Sigma_{8d}, Y^{8\lambda}).$$
\end{thm}

\begin{proof}By Proposition \ref{prop: dimext2ismultinsummands}, it is enough to show the two equalities:

\begin{eqnarray}
\label{eq: ext2stabilityfromext1}
[L(2) \otimes S^{8d-2}(V) : L(8\lambda)] &=& [L(2) \otimes S^{4d-2}(V) : L(4 \lambda)] \\
\label{eq: stabilityequtionforext2} [S^2(V^{(1)}) \otimes
S^{8d-2}(V): L(8\lambda)]&=& [S^2(V^{(1)}) \otimes S^{4d-2}(V):
L(4\lambda)].
\end{eqnarray}
Equation \ref{eq: ext2stabilityfromext1} was obtained in the proof
of Theorem \ref{thm: stability result} so we show Equation \ref{eq:
stabilityequtionforext2} holds.

 We know $[S^2(V^{(1)})]=[L(4)] +[L(2^2)]$.  To show the multiplicity of $L(4\lambda)$ coming from the  $L(4)$  term is equal to that of $L(8 \lambda)$, the proof proceeds exactly as  the proof of Theorem \ref{thm: stability result}, everything is just twisted once.

Now suppose $L(2^2) \otimes L(\mu)$ has a constituent $L(4
\lambda)$, where $\mu=2^{a_1} + 4^{a_2}+ \cdots$. By Lemma \ref{lem:
tensorproducts11x11111},  this only happen when $a_1=2$, in which
case $[L(2^2) \otimes L(2^2)] =[L(4^2)]+2[L(4,2^2)]+2[L(2^4)]$. Thus
the $L(4\lambda)$ which occur are exactly the constituents of
$L(4^2) \otimes L(4^{a_2}) \otimes L(8^{a_3})\cdots.$

Next suppose $L(2^2) \otimes L(\mu)$ has a constituent $L(8
\lambda)$, where $\mu=(2^{a_1}) + (4^{a_2})+(8^{a_3})+ \cdots$.
Arguing as above we see that we must have $a_1=a_2=2$. By Lemma
\ref{lem: tensorproducts11x11111}, $L(2^2) \otimes L(2^2) \otimes
L(4^2)$ has $L(8^2)$ as its only term divisible by 8. Thus the
$L(8\lambda)$ which occur are the constituents of $L(8^2) \otimes
L(8^{a_3}) \otimes \cdots.$, and the result is proved as in
\eqref{eq: first stability equationcomputation}.

\end{proof}

Proposition \ref{prop: dimext2ismultinsummands} gives three modules
to analyze, namely $L(2) \otimes S^{d-2}(V)$, $L(4) \otimes
S^{d-4}(V)$ and $L(2^2) \otimes S^{d-4}(V)$. For a given $2\lambda$,
we must find the number of times $L(2\lambda)$ occurs in each. For
$L(2) \otimes S^{d-2}(V)$, we know that $L(2\lambda)$ occurs the
dimension of $\HH^1(\sd, Y^{2\lambda})$ times, which was already
determined. Next we must determine the constituents of $L(4) \otimes
S^{d-4}(V)$ and $L(2^2) \otimes S^{d-4}(V).$ The former is
straightforward:

\begin{lem}
\label{lem: constituents of L(4)xSd-4(V)}
 Suppose $L(2\lambda)$ is a constituent of $L(4) \otimes S^{d-4}(V)$. Then $2\lambda$ is of the form $(2^{a_1})+ 4\mu$ for some $a_1 \geq 0$ and $\mu \neq \emptyset$. The multiplicity of $L(2\lambda)$ is the dimension of $\HH^1(\Sigma_{2c}, Y^{2\mu})$, given by Theorem \ref{thm: determineext1(k,Ylambda)}, where $c=d/2 - a_1$.
 \end{lem}
 \begin{proof}
 By Proposition \ref{prop: Dotytheoreminchar2}, we are considering constituents of tensor products of the form $L(4) \otimes L(2^{a_1}) \otimes L(4^{a_2}) \otimes \cdots$. Rearranging, we get
 \begin{equation}
 \label{eq:case2ext2}
 L(2^{a_1}) \otimes (L(2) \otimes L(2^{a_2}) \otimes L(4^{a_2}) \otimes \cdots )^{(1)}
 \end{equation}
 The constituent $L(2\mu)$ appears in $L(2) \otimes L(2^{a_2}) \otimes L(4^{a_2}) \cdots$ exactly the dimension of $\HH^1(\Sigma_{2c},Y^{2\mu})$ times, by Proposition \ref{prop: Ext1ismultinVtimesSd-2(V)}, so the result follows.

 \end{proof}

Finally we turn to the constituents of $L(2^2) \otimes S^{d-4}(V)$
of the form $L(2\lambda)$. Thus we must consider
$$L(2^2) \otimes L(2^{a_1}) \otimes L(4^{a_2}) \otimes  \cdots.$$ Suppose $a_1 \neq 2$. We know from Lemma \ref{lem: tensorproducts11x11111} that these constituents will be all of the form $L((\mu) +(4^{a_2})+ (8^{a_3}) + \cdots)$ where $\mu$ is of the form $(4,2^a), (4^2,2^a)$ or $(2^a)$, and the multiplicities depending on the congruence class of $a$ mod 4. When $a_1=2$ we also have an $L(4,4)$ term in the $L(2,2) \otimes L(2,2)$, so we also get higher twists of the constituents above.

We  now have all the information  necessary to completely determine $\HH^2(\Sigma_d, Y^\lambda)$.

\begin{thm}
\label{thm: ext2foryoungmodule} Let $p=2$ and let
$$\lambda=\lambda_{(0)} + 2\lambda_{(1)} + 2^s\lambda_{(s)} + \cdots
+ 2^r\lambda_{(r)}$$ be the 2-adic expansion of $\lambda$, where
$\lambda_{(s)}$ is nonempty unless $\lambda=\lambda_{(0)} +
2\lambda_{(1)}$.

\begin{enumerate}
  \item[(a)] If $\HH^2(\Sigma_d, Y^\lambda) \neq 0$ then $\lambda_{(t)}$ is of the form $(1^{a_t}), a_t \geq 0$ for all $t\neq 1,s$.
  \item[(b)] Suppose the $\lambda_{(t)}$'s are as in $(a)$. Then the choices of $\lambda_{(1)}$ and $\lambda_{(s)}$ which  give  nonzero $\HH^2(\sd, Y^\lambda)$ are given in  Table \ref{table: ext2youngmodules}, together with the dimension.\\

\begin{table}[here]

\begin{tabular}{|c|ccccccccc|}

 \hline
 $\lambda_{(s)}  \setminus \lambda_{(1)}$&\rule{0cm}{0.4cm}
  $1$ & $1^2$ & $1^{4a-1}$ & $1^{4a}$ & $1^{4a+1}$ & $1^{4a+2}$ &
 $21^{2c-1}$ & $21^{2c}$ & $221^a$ \\    \hline
 $ 1^{2b}$ &\rule{0cm}{0.4cm} $3$ & $5$ & $4$ & $6$& $5$ & $7$ & $2$ & $3$ & $1$\\
  $1^{2b-1}$ &$2$& $4$ & $3$ & $5$ & $4$& $6$ & $2$ & $3$ & $1$\\
  $21^b$ & $1$ & $1$ & $1$ & $1$ & $1$ & $1$ & $0$ &$ 0$ &$0$ \\
  $\emptyset$ & $1$&$3$ & $2$ & $4$ & $3$ & $5$ & $2$ & $3$ & $1$\\
  \hline
\end{tabular}

 \vspace{0.2in}
\begin{tabular}{|c|ccccccccccc|}

 \hline
 $\lambda_{(1)}  \setminus \lambda_{(s)}$& $\emptyset$ &\rule{0cm}{0.4cm} $1$ & $1^2$ & $1^{4a-1}$ & $1^{4a}$ & $1^{4a+1}$ & $1^{4a+2}$ & $1^{4a+3}$ & $21^{2c-1}$ & $21^{2c}$ & $221^c$\\    \hline
$\emptyset$ & $0$ & $2$ & $5$ & $3$ & $6$\rule{0cm}{0.4cm} & $4$ & $7$ & $3$ & $3$&$4$ & $1$\\
  \hline
\end{tabular}
\vspace*{0.1in} \caption{Nonzero $\dim \HH^2(\Sigma_d,Y^\lambda)$ sorted by
$\lambda_{(1)}$ and $\lambda_{(s)}$ , $a, b, c \geq 1.$} \label{table: ext2youngmodules}
\end{table}
\end{enumerate}

\end{thm}

\begin{proof}
Suppose $\HH^2(\Sigma_d,Y^\lambda) \neq 0$. Then $\lambda_{(0)}$
must be of the desired form by Theorem \ref{thm: Young modules with
no cohomology}. In this case we can assume $\lambda_{(0)} =
\emptyset$ by Theorem \ref{thm: oktostrip1^aoff}. So suppose
$\lambda=2\mu$. Proposition \ref{prop: dimext2ismultinsummands} and
the discussion before Lemma \ref{lem: constituents of L(4)xSd-4(V)}
tell us that $L(\lambda)$ is a constituent of at least one of  $L(2)
\otimes S^{d-2}(V)$, $L(4) \otimes S^{d-4}(V)$ or $L(2^2) \otimes
S^{d-4}(V)$.
 It is a constituent of $L(2) \otimes S^{d-2}(V)$ precisely when $H^1(\Sigma_d, Y^\lambda) \neq 0$, and the $\lambda$ that occur are given by Theorem \ref{thm: determineext1(k,Ylambda)}.

 The second case is handled in Lemma \ref{lem: constituents of L(4)xSd-4(V)}, we see that $\lambda$ must have the form $(2^{a_1})+4\mu$ where $H^1(\Sigma_s, Y^{2\mu}) \neq 0$.

 The constituents in the final case are determined just after the proof of Lemma \ref{lem: constituents of L(4)xSd-4(V)}. The three cases together give part (a) of the theorem.

Part (b) is just a matter of applying what we have already figured
out. For example, we explain the ``7" in the top row of Table \ref{table:
ext2youngmodules}. From Proposition \ref{prop:
dimext2ismultinsummands}, we need to figure out the multiplicity of
$L(2\lambda)$ in
$$L(2) \otimes S^{d-2}(V) \oplus S^2(V^{(1)}) \otimes S^{d-4}(V)$$
where $2\lambda$ is of the form $((2^{4a+2})+(4^{2b})+(8^{a_3}) +
(16^{a_4}) + \cdots)$, i.e.

$$L(2\lambda) \cong L(2^{4a+2})\otimes L(4^{2b}) \otimes L(8^{a_3}) \otimes \cdots.$$

Since $4a+2$ is even, $L(2\lambda)$ occurs twice in $L(2) \otimes
S^{d-2}(V)$. Now since $2b$ is even, Lemma \ref{lem: constituents of
L(4)xSd-4(V)} tells us that $L(2\lambda)$ appears twice in $L(4)
\otimes S^{d-4}(V)$.

Finally consider $L(2^2) \otimes S^{d-4}(V)$. $L(2\lambda)$ will
only appear in the $L(2^2) \otimes L((2^{4a})+(4^{a_2}) + \cdots)$
term, and the multiplicity will be 3 by \eqref{eq: tensorproduct
1x1111}. Thus we have a total multiplicity of seven. The other cases
are similar.
\end{proof}

\subsection{}Before concluding this section we make an observation about the versions of
Propositions \ref{prop: Ext1ismultinVtimesSd-2(V)} and \ref{prop:
dimext2ismultinsummands} which would appear in calculating higher
degree cohomology groups. We will apply it the next section.

\begin{prop}
\label{prop: twistedpartboundeddegree} Fix $i>0$ and let $p$ be
arbitrary. There is a finite list of partitions $\{\mu_s \vdash j_s
\mid 1 \leq s \leq t_i\}$, not necessarily distinct, such that for
$\lambda \vdash d$,
$$\dim \HH^i(\Sigma_d, Y^\lambda)= \sum_{s=1}^{t_i} [L(p\mu_s) \otimes S^{d-2j_s}(V):L(\lambda)]$$
and where each $j_s\leq i$.

\end{prop}

For example if $p=2$ and  $i=1$ the list is $\{(1)\}$ (cf.
Proposition \ref{prop: Ext1ismultinVtimesSd-2(V)}) while for $i=2$
the list is: $\{(1), (2), (1,1)\}$ (cf. Proposition \ref{prop:
dimext2ismultinsummands}).

\begin{proof} From Theorem \ref{thm:shiftedfinalanswer for p=2 and p oddL}, we see that there are only finitely many shapes for the ``Q" part of monomials in $\oplus_{d \geq 0}\HH_\bu(\Sigma_d, V^{\otimes d})$ which have degree $i$.  Thus the  $GL_d(k)$-summands that contribute to degree $i$ are
each (by Corollary \ref{corgivingstructure}) of the form $S^{d-a}(V) \otimes M^{(1)}$, where $M$ is a possibly complicated tensor product of symmetric and exterior powers of $V$, and where there are only finitely many choices for $M$. The list then consists of
the constituents of the various $M$ which arise.
\end{proof}

\section{Generic cohomology for Young modules}
\label{sec: Generica cohomology for Young modules}
\subsection{} From the preceding sections one could imagine even more elaborate formulas for $\HH^3(\sd, Y^\lambda)$ and higher degrees. As $i$ grows, the number of possible monomial shapes continues to grow, as does the number of tensor products one must compute. Rather than continuing in this direction, we instead observe that the stability behavior  for low degree cohomology exists in all degrees and all characteristics. We will also prove that, for a given $i$, only a finite number of tensor product calculations are required to produce formulas for $\HH^i(\Sigma_d, Y^\lambda)$ valid for $d$ and $\lambda$ arbitrary.

In the process of computing low degree cohomology, we have obtained (cf. Propositions \ref{prop:
Dotytheoreminchar2}, \ref{prop: dimensionofHomkYlambda}, and
Theorems \ref{thm: stability result},  \ref{thm: stability
resultforext2}) some ``stability"  results which we collect below:

\begin{prop}
\label{prop: stabilitytheoremstogether} Let $\lambda \vdash d $ and
$p=2$.
\begin{itemize}
  \item[(a)] $\HH^0(\Sigma_{d},Y^\lambda) \cong \HH^0(\Sigma_{2d}, Y^{2\lambda})$
  \item[(b)] $\HH^1(\Sigma_{2d},Y^{2\lambda}) \cong \HH^1(\Sigma_{4d}, Y^{4\lambda})$
  \item[(c)]  $\HH^2(\Sigma_{4d},Y^{4\lambda}) \cong \HH^2(\Sigma_{8d}, Y^{8\lambda})$
\end{itemize}
\end{prop}

We will show that this stability behavior extends to $\HH^i$ for all
$i \geq 0$ and in any characteristic. We begin with a generalization
of Proposition \ref{prop: Dotytheoreminchar2}:

\begin{prop}
\label{prop: Dotytheoremanychar} Let $\lambda \vdash s $ have a
$p$-adic expansion $$\lambda = \sum_{i=0}^mp^i\lambda_{(i)}$$ where
each $\lambda_{(i)}$ is $p$-restricted. Then $L(\lambda)$ is a
constituent of $S^s(V)$ if and only if each $\lambda_{(i)}$ is of
the form $((p-1)^{a_i}, b_i)$ for $a_i \geq 0$ and $0 \leq b_i
<p-1$.
\end{prop}

\begin{proof}The argument is the same as for Proposition \ref{prop: Dotytheoreminchar2}, except each column in the addition must have all $p-1$'s pushed to the top, a single entry between zero and $p-2$, and then all zeros.
\end{proof}

To prove the general stability result, we need to understand the
highly twisted simple modules which occur in $L(\lambda) \otimes
S^t(V)$, i.e. modules of the form $L(p^c\mu)$ for large $c$. The
next lemma shows that, for $c$ large enough, $\mu$ is completely
determined.

\begin{lem}
\label{lemma: 2cmusintensorproduct}
 Let $\lambda=(\lambda_1, \lambda_2, \ldots, \lambda_r)$ with $\lambda_1 \leq p^c$. Suppose that $\rho=\sum_{i=0}^{c-1}p^i\rho_{(i)} \vdash a$ is such that $[S^a(V) : L(\rho)] \neq 0$. Consider $M=L(\lambda) \otimes L(\rho)$.
Then $[M: L(p^c\mu)] \neq 0$ implies that $\mu=(1^w)$ for some $w
\geq 1$.
\end{lem}

\begin{proof}This follows from considering the highest weight $\tau$ which occurs in $M$.  Proposition \ref{prop: Dotytheoremanychar} tells us each $\rho_{(i)}$ has first part at most $p-1$. Thus $\tau$ has first part
\begin{eqnarray*}
  \tau_1 &\leq&  \lambda_1+p-1 + p(p-1) + \cdots + p^{c-1}(p-1) \\
   &=& \lambda_1 + p^c-1 \\
   & \leq & p^c+p^c-1 \\
   & <&2p^c
\end{eqnarray*}
Thus $M$ has no dominant weight of the form $p^c\mu$ unless
$\mu_1=1$, and the result follows.

\end{proof}

We need another lemma about twisted constituents in tensor products:

\begin{lem}
\label{lemma:lemmaabout1aotimes1b}Let $\rho=((p-1)^a, b)$ with $0
\leq b <p-1$. Then $L(1^w) \otimes L(\rho)$ contains no constituents
of the form $L(p\mu)$ unless $b=0$ and $a=w$, in which case the only
such constituent is a single copy of $L((p)^a)=L((p,p,\ldots, p))$.
\end{lem}
\begin{proof}This is clear from considering the weight spaces in $L(1^w) \otimes L(\rho)$, as there is not even a nonzero weight space of the form $p\mu$ unless $a=w$, in which case $(p)^a$ is the highest weight, and the only dominant weight of that form. Thus there is a single copy of $L((p)^a)$.
\end{proof}

The preceding lemmas let us prove that the calculations in the proof
of Theorem \ref{thm: stability resultforext2} generalize.
Specifically we have the following generalization of \eqref{eq:
stabilityequtionforext2}.

\begin{prop}
\label{prop: general stability of constituents} Suppose
$\lambda=p\sigma \vdash a$ with $\lambda_1 \leq p^c$, and $\tau
\vdash d$. Then:

\begin{equation}\label{eq: generalstabilityequation}
[L(\lambda) \otimes S^{p^cd-a}(V): L(p^c\tau)]=[L(\lambda) \otimes
S^{p^{c+1}d-a}(V):L(p^{c+1}\tau)].
\end{equation}
\end{prop}

\begin{proof}
We explain how occurrences of $L(p^{c+1}\tau)$ in the right side of
\eqref{eq: generalstabilityequation} and $L(p^c\tau)$ in the left
hand side correspond bijectively.  Since we are assuming
$\lambda=p\sigma$, twisted constituents of $L(\lambda) \otimes
S^{p^{c+1}d-a}(V)$  come from modules of the form:

\begin{equation}
\label{eq: bigtensorproduct} L(\lambda) \otimes L(\mu_{(1)})^{(1)}
\otimes L(\mu_{(2)})^{(2)} \otimes \cdots \otimes L(\mu_{(c)})^{(c)}
\otimes L(\mu_{(c+1)})^{(c+1)} \otimes \cdots.
 \end{equation}
where $\mu_{(i)}=((p-1)^{a_i},b_i)$ as in Prop. \ref{prop:
Dotytheoremanychar}.

 Lemma \ref{lemma: 2cmusintensorproduct} implies that
 $$L(\lambda) \otimes L(\mu_{(1)})^{(1)} \otimes L(\mu_{(2)})^{(2)} \otimes \cdots \otimes L(\mu_{(c-1)})^{(c-1)}$$ has no constituents of the form $p^c\rho$ unless $\rho=1^w$.
 Then Lemma \ref{lemma:lemmaabout1aotimes1b} implies that constituents in \eqref{eq: bigtensorproduct} of the form $L(p^{c+1}\tau)$ occur only when $\mu_{(c)}=((p-1)^w)$.

Thus we end up counting occurrences of $L(p^{c+1}\tau)$ in

 \begin{equation}
 \label{eq:stabcomp}
 L((p^{c+1})^w) \otimes L(\mu_{(c+1)})^{(c+1)} \otimes L(\mu_{(c+2)})^{(c+2)}\otimes \cdots.
 \end{equation}

 If we repeat the analysis above for counting occurrences of $L(p^c\tau)$ in the left hand side of \eqref{eq: generalstabilityequation}, we end up counting occurrences of $L(p^c \tau)$ in
    $$L((p^{c})^w) \otimes L(\mu_{(c)})^{(c)} \otimes L(\mu_{(c+1)})^{(c+1)}\otimes \cdots.
$$
 which is the same as \eqref{eq:stabcomp} by the STPT.

\end{proof}

\subsection{} We can now state our main stability theorem.

\begin{thm}
\label{thm: Main stability result} Fix $i>0$ and let $p$ be
arbitrary. Then there exists $s(i)>0$ such that for any $d$ and
$\lambda \vdash d$ we have
$$\HH^i(\Sigma_{p^{a}d},Y^{p^{a}\lambda}) \cong \HH^i(\Sigma_{p^{a+1}d},Y^{p^{a+1}\lambda}) $$ whenever $a \geq s(i)$.
\end{thm}

\begin{proof}This follows from Propositions \ref{prop: twistedpartboundeddegree} and \ref{prop: general stability of constituents}.
\end{proof}

\begin{rem}
\label{rem: stabinchar2} Suppose $p=2$ and let $2^{a-1} <i \leq
2^a$. When calculating $\HH^i$ we must determine the multiplicities
in modules of the form $M^{(1)} \otimes S^{d-a}(V)$, where the
$M^{(1)}$ is determined by the shapes of the various monomials
contributing to degree $i$. One easily sees that if $[M: L(\mu)]
\neq 0$, then $\mu$ has degree at most $2i$. In particular $\mu_1
\leq 2^{a+1}$, and so by Prop. \ref{prop: general stability of
constituents}, we can choose $s(i)=i+1$ in Theorem \ref{thm: Main
stability result}.
\end{rem}

 \begin{rem} The choice of $s(i)$ in Remark \ref{rem: stabinchar2}, can be seen to be best possible in all degrees, in the sense that the stability would not hold for $2^i\lambda$. For example we observe that in characteristic two the stability for $\HH^3$ does not work for $Y^{4\lambda}$ and $Y^{8\lambda}$ as the  following example demonstrates.
 \end{rem}

\begin{exm} Let $p=2$. Then \begin{eqnarray*}
             \dim\HH^3(\Sigma_{12},Y^{(4^3)})&=&7. \\
              \dim\HH^3(\Sigma_{24},Y^{(8^3)}) &=& 8.
           \end{eqnarray*}
\end{exm}

The difference comes from the $L(6) \otimes L(2) \otimes L(8^2)$
term in $L(6) \otimes S^{18}(V)$. It contributes an $L(8^3)$ for
which there is no corresponding contribution of an $L(4^3)$ inside
$L(6) \otimes S^6(V)$.

Theorem \ref{thm: Main stability result} is highly reminiscent of
the generic cohomology results of \cite{CPSvdK}. In the generic
cohomology setting where $G$ is a reductive group and $M$ a finite
dimensional $G$-module, one has a series of injective maps
\cite[II.10.14]{Jantzen}
$$\HH^i(G,M) \rightarrow \HH^i(G,M^{[1]})\rightarrow \HH^i(G,M^{[2]}) \rightarrow \cdots .$$
In \cite{CPSvdK} it was shown that this sequence stabilizes to the
\emph{generic cohomology }of $M$. In our setting then, one might
expect injections $\HH^i(\Sigma_d, Y^\lambda) \rightarrow
\HH^i(\Sigma_{pd}, Y^{p\lambda})$ which stabilize. The following
example shows there are not necessarily injections and the
dependence on the degree is quite crucial.

\begin{exm}
\label{example: injectionfails}In characteristic two it follow from
Theorem \ref{thm: determineext1(k,Ylambda)} that:

$$\HH^1(\Sigma_8, Y^{(5,3)}) \cong k, \,\,\, \HH^1(\Sigma_{16}, Y^{(10,6)})=0.$$

\end{exm}

\subsection{}
We only had to perform a small number of tensor product calculations
to get a formula for $\HH^2(\Sigma_d, Y^\lambda)$ that was valid for
any $d$ or $\lambda$. The answer was given in terms of the $2$-adic
expansion of $\lambda$ and only finitely many possibilities occurred
(cf. Table \ref{table: ext2youngmodules}). Suppose one wanted a
formula for $\HH^i(\Sigma_d, Y^\lambda)$. One would have to compute
the tensor products which appear in Proposition \ref{prop:
twistedpartboundeddegree}, and try to determine the multiplicity of
$L(\lambda)$. However Lemma \ref{lemma: 2cmusintensorproduct} tells
us, essentially, that the appearance of $L(\lambda)$ will be
determined by the beginning of the $p$-adic expansion of $\lambda$.
That is, if we compute
$$L(\tau) \otimes L(\mu_{(1)})^{(1)} \otimes L(\mu_{(2)})^{(2)} \otimes \cdots \otimes L((\mu_{(c)})^{(c)})$$
 for $c$ large enough, we get a finite list of simple modules which are not of the form $L(p^c \mu)$, and then some copies of $L(p^{c}( 1^w))$. Consequently we can state the following:

\begin{thm}
\label{thm:fixedbound} Fix $i>0$ and let $p$ be arbitrary. Then
there is a $t>0$ such that computing the multiplicities $[L(p\mu_s)
\otimes S^t(V) : L(\lambda)]$ for each $\mu_s$ appearing in
Proposition \ref{prop: twistedpartboundeddegree} is enough to
determine $\HH^i(\Sigma_d, Y^\lambda)$ for any $d$ and any $\lambda
\vdash d$. Furthermore, for fixed $i$ the dimension of
$\HH^i(\Sigma_d, Y^\lambda)$ is bounded uniformly, independent of
$d$ and $\lambda$.
\end{thm}

Theorem \ref{thm:fixedbound} says that for each fixed $i$ there is
some table like Table \ref{table: ext2youngmodules}, just larger and
depending on more initial terms in the $p$-adic expansion of
$\lambda$.

\section{Cohomology of Permutation Modules}
\label{sec: Cohomology of permutation modules}

\subsection{}
We have seen that $\HH^i(\Sigma_d, Y^\lambda)$ is the multiplicity
of the simple module $L(\lambda)$ in an explicitly given
$GL_d(k)$-module. In this section we will see that the cohomology
$\HH^i(\Sigma_d, M^\lambda)$ of the permutation module $M^\lambda$
is determined by the same $GL_d(k)$-module, however this time by the
$\lambda$-weight space of the module.

Substitute  $M=S^\lambda(V)$ and $N=k$ into the spectral sequence
\eqref{eq: spectralsequenceTorversion}. Since $S^\lambda(V)$ is
injective as an $S(d,d)$-module, the spectral sequence collapses and
we get:

\begin{prop}
\label{prop: homstatementpermumodulecoho} Let $p$ be arbitrary. Then
$$\dim \HH^i(\Sigma_d , M^\lambda)= \dim
\Hom_{GL_d(k)}(\HH_i(\Sigma_d, V^{\otimes d}), S^\lambda(V)).$$
\end{prop}

By \cite[2.1-2.3]{DW} the dimension of $\Hom_{GL_d(k)}(U,
S^\lambda(V))$ is just the dimension of the $\lambda$-weight space
$U_\lambda$. Thus we have the following result, where we restate the
corresponding Young module result \ref{thm:multl of simples in
Ljgivescohomology}(b) for comparison. Cohomology for $Y^\lambda$ is
controlled by the composition multiplicities in a certain $GL_d(k)$
module whereas for the permutation module it is controlled by weight
spaces in the \emph{same} module.

\begin{thm}Let $p$ be arbitrary.
\label{thm: cohomology of permutation modules} \
\begin{enumerate}
  \item[(a)] $\dim \HH^i(\Sigma_d,  M^\lambda)= \dim \HH_i(\Sigma_d, V^{\otimes d})_\lambda$.
  \item[(b)] $\dim \HH^i(\Sigma_d, Y^\lambda)= [\HH_i(\Sigma_d, V^{\otimes d}): L(\lambda)]$.
\end{enumerate}
\end{thm}

\begin{rem} The weight space dimensions in Theorem \ref{thm: cohomology of permutation modules}(a) are in principal known, as we have described the module $\HH_i(\Sigma_d, V^{\otimes d})$ as a tensor product of modules with known weight space decompositions. In principal one can calculate $\HH^i(\Sigma_d, M^\lambda)$ from Nakaoka's work and repeated application of the Kunneth theorem. However we find Theorem \ref{thm: cohomology of permutation modules} a much more conceptual and pleasing interpretation.
\end{rem}

\end{document}